\documentclass{amsart}

\usepackage{graphicx}
\usepackage{float}
\usepackage{enumerate}
\usepackage{hhline}

\restylefloat{figure}
\usepackage{wrapfig}
\usepackage{amsmath}
\usepackage{amssymb}
\usepackage{amsthm}
\usepackage{hyperref}
\usepackage{pinlabel}
\usepackage{array}
\usepackage{complexity}
\usepackage{epstopdf}
\usepackage{appendix}

\usepackage{fancyhdr}
\newtheorem{theorem}{Theorem}[section]
\newtheorem{lemma}[theorem]{Lemma}
\newtheorem*{Inequality Lemma}{Inequality Lemma}
\newtheorem*{Compactness Lemma}{Compactness Lemma}

\newtheorem*{Lower Bound Theorem}{Lower Bound Theorem}
\newtheorem*{Upper Bound Theorem}{Upper Bound Theorem}
\newtheorem*{Non-Classifiability Theorem}{Non-Classifiability Theorem}
\newtheorem*{Essential Membership Theorem}{Essential Membership Theorem}
\newtheorem*{Generic Word Theorem}{Generic Word Theorem}

\newtheorem*{Complexity Theorem}{Complexity Theorem}

\newtheorem{proposition}[theorem]{Proposition}
\newtheorem{corollary}[theorem]{Corollary}
\newtheorem{conjecture}[theorem]{Conjecture}

\newtheorem{problem}[theorem]{Problem}
\newtheorem*{claim}{Claim}
\theoremstyle{definition}
\newtheorem{definition}[theorem]{Definition}

\theoremstyle{remark}
\newtheorem{remark}[theorem]{Remark}
\newtheorem{example}[theorem]{Example}

\numberwithin{equation}{subsection}

\newcommand{\NN}{\mathbb{N}}
\newcommand{\ZZ}{\mathbb{Z}}
\newcommand{\QQ}{\mathbb{Q}}
\newcommand{\RR}{\mathbb{R}}

\newcommand{\ComplexityFontII}[2]{\ComplexityFont{#1} \ComplexityFont{#2}}
\newcommand{\ComplexityFontIII}[3]{\ComplexityFontII{#1}{#2} \ComplexityFont{#3}}

\newcommand{\ESSENTIAL}{\ComplexityFont{ESSENTIAL}}

\newcommand{\SUBSET}{\ComplexityFontII{SUBSET}{SUM}}
\newcommand{\SUBSETT}{\ComplexityFontII{(SUBSET}{SUM)}}
\newcommand{\SUBSETP}{\ComplexityFontII{SUBSET}{SUM'}}

\newcommand{\SUBSETTP}{\ComplexityFontII{(SUBSET}{SUM')}}

\newcommand{\SMALLSCL}{\ComplexityFontII{SMALL}{SCL}}
\newcommand{\SMALLSCLT}{\ComplexityFontII{(SMALL}{SCL)}}
\newcommand{\VARSUBSETP}{\ComplexityFontIII{VAR}{SUBSET}{SUM'}}

\newcommand{\MIXEDSUBSETP}{\ComplexityFontIII{MIXED}{SUBSET}{SUM'}}

\newcommand{\MIXEDSUBSETTP}{\ComplexityFontIII{(MIXED}{SUBSET}{SUM')}}
\newcommand{\VARSUBSETTP}{\ComplexityFontIII{(VAR}{SUBSET}{SUM')}}
\newcommand{\COSUBSET}{\ComplexityFontII{coSUBSET}{SUM}}

\newcommand{\COSUBSETT}{\ComplexityFont{(coSUBSET-SUM)}}

\DeclareMathOperator{\cl}{cl}
\DeclareMathOperator{\scl}{scl}

\DeclareMathOperator{\inflow}{inflow}
\DeclareMathOperator{\outflow}{outflow}
\DeclareMathOperator{\indegree}{indegree}
\DeclareMathOperator{\outdegree}{outdegree}
\DeclareMathOperator{\abst}{abst}

\DeclareMathOperator{\supp}{supp}
\DeclareMathOperator{\conv}{conv}
\DeclareMathOperator{\cone}{cone}

\title{On the Complexity of Sails}
\date{}

\author{Lukas Brantner\\ (Appendix \lowercase{by} Freddie Manners)}
\dedicatory{To my family}

\begin{document}
\begin{abstract}
This paper analyses stable commutator length in groups $\ZZ^r\ast \ZZ^s$.

We bound $\scl$ from above in terms of the reduced wordlength (sharply in the limit) and from below in terms of the answer to an associated subset-sum type problem. Combining both estimates, we prove that, as $m$ tends to infinity, words of reduced length $m$ generically have scl arbitrarily close to $\frac{m}{4}-1$. 

We then show that, unless $\P=\NP$, there is no polynomial time algorithm to compute $\scl$ of efficiently encoded words in $F_2$.

All these results are obtained by exploiting the fundamental connection between $\scl$ and the geometry of certain rational polyhedra. Their extremal rays have been classified concisely and completely. However, we prove that a similar classification for extremal points is impossible in a very strong sense.
\end{abstract}
\maketitle
\markboth{}{}
\fancyhead[EC]{\uppercase{\footnotesize{Lukas Brantner}}}
\fancyhead[OC]{\uppercase{\footnotesize{On the Complexity of Sails}}}

\vspace*{-1.5em}
\section{Introduction} \label{sec:chapterone}
Stable commutator length (hereafter $\scl$) is a concept in geometric group theory which arises naturally in the study of least genus problems such as:
\begin{quote}
Given a topological space X and a loop $\gamma$, what is the least genus of a once-punctured, orientable surface which can be mapped to $X$ such that the boundary wraps once around $\gamma$?
\end{quote}
It transpires that the real-valued function $\scl$ gives an algebraic analogue of the (relative) Gromov-Thurston norm in topology and has deep connections to various areas of interest in modern geometry (see \cite{[3]}). The computation of scl is notoriously difficult and its distribution often mysterious, even in free groups. Important open problems in the theory of scl in such groups are to determine the image of $\scl$ (``inverse-problem"), and, more ambitiously, to find a clear relation between the outer form of a word and its scl (``form-problem").

An \textit{a priori} completely  unrelated concept ubiquitous in the theory of linear optimization is that of a (convex) polyhedron and its boundary, the \textit{sail}. If such a polyhedron P is pointed (i.e.~does not contain any line), it has a particularly simple ray-vertex-decomposition as $P=\cone(R)+\conv(V)$, where $R$ and $V$ are the \textit{finite} sets of extremal rays and points respectively (see Chapter $8$ of \cite{[1]}). Combinatorial optimization is often concerned with polyhedra whose elements represent flows, and which are given to us as the convex hulls of combinatorially distinguished flows (e.g. paths from source to sink, see Chapter $13$ in \cite{[6]}). In such cases, the description of $V$ and $R$ is a crucial step towards a complete understanding of the geometry of $P$.

These two concepts were bridged by Calegari's algorithm (see \cite{[4]}), which establishes an intricate connection between the computation of scl in groups of the form  $\displaystyle *_{i=1}^m  \ZZ^{m_i}$ and the geometry of certain \textit{rational} flow-polyhedra. The sails of these polyhedra are the unit sets of one-homogeneous functions, which one has to maximize over certain subsets in order to compute $\scl$.

There are two ways in which this link can be exploited:
the \textit{relative approach} compares the polyhedra corresponding to different words and converts geometric relations between them into numerical ones relating their $\scl$s.
In contrast, the \textit{absolute approach} uses the precise, very involved geometry of individual polyhedra to compute the $\scl$ of given words exactly.
The former technique is significantly more accessible as it does not require such a detailed analysis. Amongst other things, it has been used to prove the salient \textit{Surgery Theorem} (see Theorem $4.13$ in \cite{[4]}), which demonstrates that the scl of certain natural sequences of words converges. 

Following this method, we start off the first section of this paper by observing that certain linear-algebraic relations between exponents of words translate directly into inequalities of $\scl$ and then use this to relate the $\scl$-images of different groups of the form $\ZZ^r \ast \ZZ^s$. More importantly, we combine both of the aforementioned approaches to obtain new upper and lower bounds and use these to prove that the $\scl$ of a generic word of reduced length $m$ is close to $\frac{m}{4}-1$. 

Our lower bound implies that to prove the long-standing open conjecture that $\scl(\ZZ \ast \ZZ)\supseteq \QQ\cap[1,\infty)$, we can restrict our attention to a certain subclass of words. 

Our second main theorem shows that computing $\scl$ is hard: unless $\P = \NP$, the $\scl$ of a word cannot be determined in polynomial time.  

The second approach is more formidable, but promises more substantial progress towards a complete solution of the two guiding problems mentioned initially. An exhaustive analysis of the polyhedral geometry has been carried out in a few specific cases (see section $4.1$ in \cite{[4]}), allowing the explicit computation of scl in several infinite families of words. These partial successes raised the hope that a complete description of the polyhedra was within reach. Indeed, the first half of their ray-vertex-decomposition was found by Calegari who provided a general and simple classification of their extremal rays (Lemma 4.11 of \cite{[4]}, see page $4$).

The main result of the second section of this paper demonstrates that the next step cannot be made: a similar classification for extremal points is impossible, roughly speaking because they exhibit provably arbitrarily complicated behaviour. We conclude the paper by showing that a natural alternative description of the relevant polyhedra is infeasible from a complexity-theoretic perspective.

\subsection{Main Results}\label{sec:intro} We first use polyhedra to prove positive theorems on $\scl$, and then we provide negative results explaining why certain nice descriptions of these polyhedra cannot exist. All words are assumed to lie in the commutator subgroup of $\ZZ^\infty \ast \ZZ^\infty$\footnote{Here $\ZZ^\infty$ is free abelian group on countably many generators, which we denote by $\{a_i\}_{i\in \NN}$ in the left factor and by $\{b_i\}_{i\in \NN}$ in the right factor. For $\mathbf{x}\in \ZZ^\infty$ with components $x^{(i)}$, we write $\mathbf{a}^\mathbf{x}=a_1^{x^{(1)}}\cdot a_2^{x^{(2)}}\cdot...$, a similar expression defines $\mathbf{b}^\mathbf{x}$.}, to start in the left and to end in the right factor. This particular case comprises all words in all groups $\ZZ^r\ast \ZZ^s$. A word has reduced length $m$ if it switches $m$ times from one to the other factor of our free product. This notion generalises to all free products, and it differs from the classical wordlength, which counts the number of letters in a word. 
For the words we examine, $m=2n$ is even.

In Section \ref{background}, we define stable commutator length (\ref{sec:scl}), and then give a detailed description of Calegari's algorithm, thereby introducing relevant terminology (\ref{sec:calegari}).

In Section \ref{sec:chaptertwo}, we prove bounds on $\scl$ and the complexity of its computation. Words of \textit{reduced length} $m=2n$ in $\ZZ^\infty \ast \ZZ^\infty$ are most naturally expressed as:
$$\phi(x,y)= \mathbf{a}^\mathbf{x_1} \cdot \mathbf{b}^\mathbf{y_1} \cdot ... \cdot \mathbf{a}^\mathbf{x_n} \cdot \mathbf{b}^\mathbf{y_n}$$ for $x=\{\mathbf{x_j}\} , y=\{\mathbf{y_j}\}$ certain collections of nonzero vectors.

We start by proving that values in the set $\scl(\ZZ^\infty \ast \ZZ^\infty)\backslash \scl(\ZZ^r \ast \ZZ^s)$ cannot come from words that are ``too short":
\begin{Compactness Lemma} If $v\in \ZZ^\infty \ast \ZZ^\infty$ has reduced length $N$, its $\scl$ is already contained in the image $\scl(\ZZ^r \ast \ZZ^s)$ for all $r,s\geq N$.\end{Compactness Lemma}

More importantly, we give a lower bound on $\scl$ depending on the number of exponents we need to represent zero as a nontrivial sum (repetitions allowed). 
\begin{Lower Bound Theorem}
Let $w=\phi(x,y)\in \ZZ^\infty \ast \ZZ^\infty$ have reduced length $2n$. 

Fix $p,q \in \NN$, and assume that the following two implications hold:

If $(\lambda_j)_j \in \NN^n\backslash\{0\}$ is a vector with $\sum_j \lambda_j x^{(i)}_j =0$ for all $i$, then $\sum_j \lambda_j \geq p$.

If $(\mu_j)_j \in \NN^n\backslash\{0\}$ is a vector with $\sum_j \mu_j y^{(i)}_j =0$ for all $i$, then $\sum_j \mu_j \geq q$.

In this case, we have the inequality $\scl(w) \geq \frac{n}{2}(1-\frac{1}{p}-\frac{1}{q})$
\end{Lower Bound Theorem}

For each length, intersecting the polyhedra of all words of this length yields an upper bound on $\scl$ which is ``best possible in the limit":

\begin{Upper Bound Theorem} Write $C(m)$ for the supremum of the $\scl$ of words in $\ZZ^\infty\ast \ZZ^\infty$ of reduced length $m=2n>4$. Then this supremum is attained and satisfies
$$C(m)\leq   \left\{\begin{array}{cl} \frac{n}{2}-1, & \mbox{if $n$ odd}\\ 
\frac{n}{2}-\frac{(n-1)!-1}{n(n-2)!-2}, & \mbox{if $n$ even} \end{array}\right.$$
Moreover, given $\epsilon>0$, we have for $m$ sufficiently large: $\frac{m}{4}-1\leq C(m) \leq \frac{m}{4}-1+\epsilon$
\end{Upper Bound Theorem}

To state our main result precisely, we need to define what we mean by a \textit{``generic property"}. Recall the map $\phi$ from above, which associates a word of reduced length $m=2n$ to pairs $(x,y)$ of certain collections of vectors in the rank-$(n-1)$-module $V=\{z\in \ZZ^n|\sum_j z_j =0\}$.
\begin{definition}\label{def:generic}
Let $P$ be a property on words in the commutator of $\ZZ^\infty \ast \ZZ^\infty$. We say words of reduced length $m=2n$ \textit{generically satisfy} $P$ if there are finitely many submodules $W_1,...,W_l$ of $V$ of ranks at most $n-2$, such that whenever not all $x^{(i)}$ and not all $y^{(i)}$ lie in $\bigcup_k W_k$, the property $P(\phi(x,y))$ holds.
\end{definition} Welding the upper and the lower bound together, we conclude:
\begin{Generic Word Theorem}
Given any $\epsilon>0$, we can choose $N$ such that for all $m\geq N$, words $w$ of reduced length $m$ generically satisfy
$$ \scl(w)\in \left[\frac{m}{4}-1,\frac{m}{4}-1+\epsilon\right]$$
\end{Generic Word Theorem}
For efficiently encoded words in $F_2$, all known $\scl$ algorithms are computationally expensive. Here we elucidate that such expenditure arises not through the fault of the algorithms but from the intrinsic difficulty of the determination of $\scl$:
\begin{Complexity Theorem}
Unless $\P=\NP$, the $\scl$ of words $\phi(x,y)\in F_2$ cannot be computed in polynomial time in the input size of the vector $(x,y)$.
\end{Complexity Theorem}

In Section \ref{sec:chapterthree}, we analyse the geometry of the relevant flow-polyhedra. In order to express the main result precisely, we need to clarify what we mean by the ``abstract graph underlying a flow". 
\begin{definition} \label{MDgraphs} Take the smallest equivalence relation on the class of (finite) multi-digraphs (``MD-graphs") which is stable under subdivision of directed edges.

An MD-graph is called \textit{abstract} if it does not contain subdivided edges. Note that every MD-graph is equivalent to a unique abstract graph. 

\begin{figure}[h]
	\centering
		\includegraphics[width=0.30\textwidth]{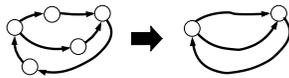}
				\caption{The abstract graph of an MD-graph}
	\label{fig:FIGURE1}
\end{figure}

The underlying abstract graph of a flow is the abstract graph of its support\footnote{The support of a flow is the digraph induced by the edges with nonzero flow.}\end{definition}

The general classification of extremal rays obtained by Calegari (see Lemma 4.11 in \cite{[4]}) implies that the abstract graphs underlying extremal rays are of an elegant simplicity - they are all isomorphic to one of the following three MD-graphs: 
\begin{figure}[h]
	\centering
		\includegraphics[width=0.50\textwidth]{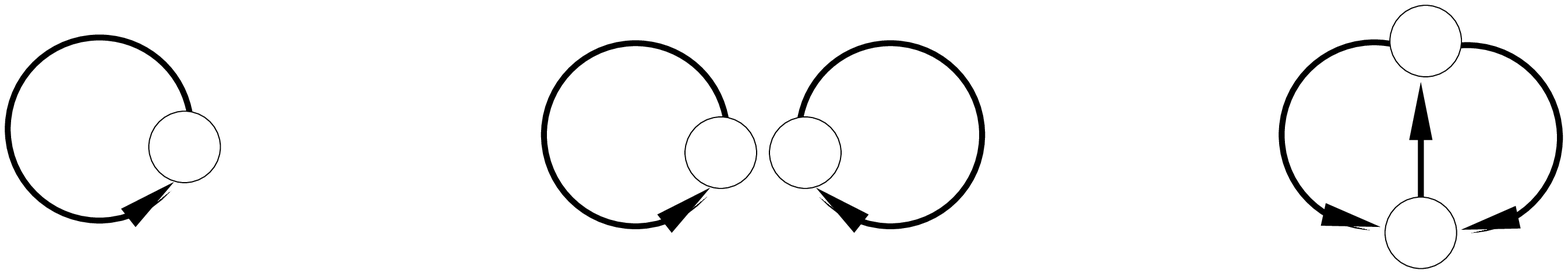}
		\caption{The three abstract graphs underlying extremal rays}
	\label{fig:FIGURE2}
\end{figure}

However, we prove that a classification of the extremal points which gives rise to any nontrivial restriction on the underlying abstract graphs cannot exist:

\begin{Non-Classifiability Theorem} For every connected, nonempty, abstract MD-graph $G$, there is an (alternating) word $w\in \ZZ \ast \ZZ$ and an extremal point $f$ of a flow-polyhedron associated to $w$ such that $G$ is the abstract graph underlying $f$.\end{Non-Classifiability Theorem}

The polyhedra in Calegari's algorithm arise as $P=\conv(D+V)$, where $V$ is the (understood) recession cone, and $D$ is an infinite integral subset of $V$. The aim is to find an efficient representation, and a very natural alternative to the vertex-ray-decomposition is the \textit{essential decomposition}: here, we use the minimal set $T\subset D$ (\textit{essential vectors}) with $T+V=D+V$ to encode $P$. Our final theorem indicates that this decomposition is computationally infeasible:
\begin{Essential Membership Theorem} The decision problem ``Given a word $w\in F_2$ and a vector $v$ in the corresponding cone, is $v$ essential?" is $\coNP-$ complete.
\end{Essential Membership Theorem}
\section{Background}\label{background}
We give a review of some basic properties of $\scl$ and relevant previous work. 

\subsubsection{Word-parametrisation $\phi$}
First, we introduce effective notation for words in the group $\ZZ^\infty \ast \ZZ^\infty$, which is the fundamental group of the wedge of two spaces. Without losing generality for our purposes, we will assume that all mentioned words are elements of the commutator subgroup of $G=A\ast B=\ZZ^\infty \ast \ZZ^\infty$ which start in $A$ and end in $B$. The following is the central invariant of words in our group; it measures how often a loop switches from one space to another:

\begin{definition} \label{def:reduced}
Every word $w$ can be written as $w=u_1 v_1 ... u_n v_n$ with $u_i \in A\backslash \{1\}$ and $v_i \in B\backslash \{1\}$. We define the \textit{reduced wordlength} (or, more concisely, reduced length) of $w$ to be $2n$. 
\end{definition}
We will examine the map $\phi$ introduced in \ref{sec:intro} in more detail. Let $U_k$ ($k\in \NN$) be copies of the space $U=\{z\in \ZZ^n| \sum_j z_j =0\}$ and define 
$$M_n=\left\{\left(z^{(1)},z^{(2)},...\right) \in \displaystyle \bigoplus_{k=1}^\infty U_k \bigg| \ \forall j\in \{1,...,n\} \ \ \exists i\in \NN: z^{(i)}_j\neq 0\right\}$$
For $z\in M_n$, we write $\mathbf{z_j}=(z_j^{(1)}, z_j^{(2)},...)$. The map $\phi$ from \ref{sec:intro} then gives a bijection between $M_n \times M_n$ and words of reduced length $2n$, which is given explicitly by:
$$\phi(x,y)= \left(a_1^{x^{(1)}_1} a_2^{x^{(2)}_1}...\right) \cdot \left(b_1^{y^{(1)}_1} b_2^{y^{(2)}_1}...\right) \cdot ... \cdot 
\left(a_1^{x^{(1)}_n} a_2^{x^{(2)}_n}...\right) \cdot \left(b_1^{y^{(1)}_n} b_2^{y^{(2)}_n}...\right)$$
For a loop $\gamma$ suitably representing $\phi(x,y)$, the integer $x^{(i)}_j$ can be geometrically interpreted as the number of times our loop walks along the $i^{th}$ generator in the left space between the $(2j-1)^{th}$ and $(2j)^{th}$ passage through $\ast$. A shifted statement holds for $y^{(i)}_j$ and the right space.

\subsection{Definition of Stable Commutator Length}\label{sec:scl}
We give a very accessible, algebraic definition and sketch an equivalent, more motivated topological one.
\begin{definition} Let $G$ be a group and $g\in [G,G]$.
The \textit{commutator length} $\cl(g)$ is defined to be the least number of commutators in $G$ whose product is $g$.

We stabilise this definition and define the \textit{stable commutator length} of $g$ to be: $$\scl(g)=\displaystyle \lim_{n\rightarrow \infty} \frac{\cl(g^n)}{n}$$
\end{definition}

There is a close link between commutator length and the least-genus problem mentioned initially, which yields a \textit{purely topological definition of $\cl$}:

\begin{proposition}
Let $(X,x)$ be a pointed topological space with fundamental group $G=\pi_1(X,x)$. Assume moreover that $\gamma$ is a based loop with homotopy class $g$. Then $\cl(g)$ is the least genus of a once-punctured, orientable, compact, and connected surface $S$ which can be mapped to $X$ such that $\partial S$ wraps once around $\gamma$. 
This follows directly from well-known classification of compact surfaces.
\end{proposition}

This result can be extended to obtain a similar \textit{topological definition of $\scl$}. It describes $\scl$ as a measure of how simple a surface (rationally) bounding a given loop can be, where the meaning of ``simple" is slightly tweaked:

A map $f:S\rightarrow X$ from a compact orientable surface $S$ is called \textit{admissible for a loop $\gamma$} if $f$ wraps the boundaries of $S$ around $\gamma$.
To such a map, we associate the quotient $\frac{\sum_j|\min(\chi(S_j),0)|}{2n}$, where $\chi(S_j)$ is the Euler characteristic of the connected components $S_j$ of $S$, and $n$ is the degree with which $f$ wraps $\partial S$ around $\gamma$. Then $\scl([\gamma])$ is given by the infimum of this quotient over all admissible maps $f$. 

Moreover, $\scl$ can be extended to homologically trivial chains on our group $G$. One can use this to continuously extend $\scl$ to the group $B_1(G)$ of $1-$boundaries in the real group homology of $G$. In many relevant cases, this extension even descends to a norm on a suitable quotient of $B_1(G)$.

However, the precise formulation of both of these definitions requires more technical care and we therefore refer the reader to the sections $2.1$ and $2.6$ in \cite{[3]}. We also recommend section $2.4$, which establishes a close connection between $\scl$ and bounded cohomology.

\subsection{Calegari's Algorithm}\label{sec:calegari}
This algorithm enables the computation of $\scl$ in free products of free abelian groups. For the sake of notational convenience, we will restrict ourselves to the specific case of two factors and to words rather than chains.
Our group is then the fundamental group of the wedge $X$ of two tori, hence we can represent its elements by loops in these spaces.
\begin{figure}[htbp]
	\centering
		\includegraphics[width=0.35\textwidth]{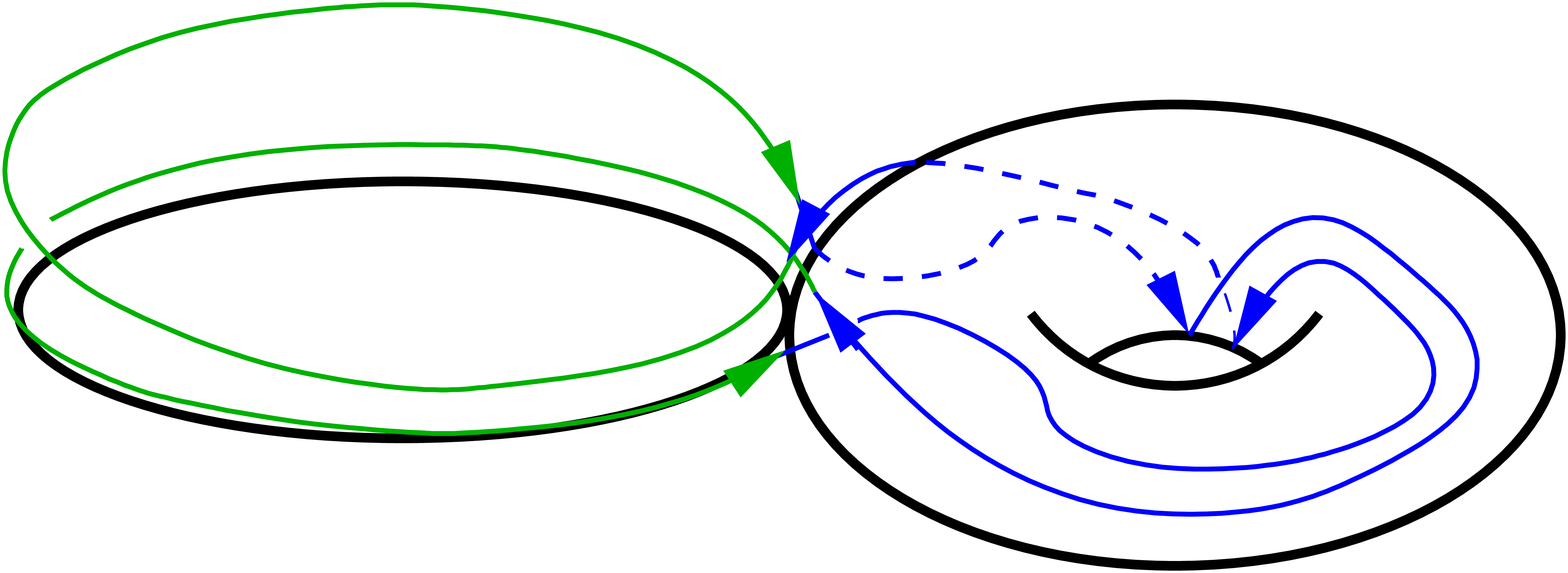}
		\caption{Loop representing $w=a_1 b_1 b_2 a_1^{-1} b_1^{-1} b_2^{-1} \in \ZZ\ast \ZZ^2$ with $\scl(w)=\frac{1}{2}$}
	\label{fig:CT}
	
\end{figure}

The algorithm proceeds in three steps: To a given word $\phi(x,y)$ of length $m=2n$, it first associates two complete digraphs on $n$ vertices. Special flows on these two graphs define two polyhedral cones equipped with $1-$homogeneous functions. The sum of these functions then has to be maximized over a subset to compute $\scl$.

\subsubsection{Sketch of proof}
The proof of Calegari's algorithm exploits the topological nature of $\scl$. Let $\gamma$ be a loop in $X$ which nicely represents $\phi(x,y)$. Given an admissible map $f:S\rightarrow X$, we cut our surface $S$ along the preimage $f^{-1}(\{\ast\})$ of the gluing point into two \textit{simple components}. Therewith, we decouple the left and the right half of our loop temporarily. 
The combinatorics of the boundaries of the two simple components give rise to a pair of flow-vectors $(v_A,v_B)$ on $n$ vertices. This pair carries all the $\scl$-relevant information we can extract from $f$. Homological triviality of the left and right half of $\gamma$ imply that $(v_A,v_B)$ lies in the Cartesian product of two polyhedral cones $V(x)$, $V(y)$. These cones define the crucial flow-polyhedra $P(x)$, $P(y)$, whose boundaries are the unit sets of the $1-$homogeneous Klein-functions $\kappa_x$, $\kappa_y$. A detailed analysis finally shows that $\scl$ can be computed by maximizing $\kappa_x+\kappa_y$ over a certain compact subset of $V(x)\times V(y)$. 

In the remainder of this section, we will give a precise formulation of the very technical terms used in this sketch. We will give entirely self-contained definitions, which do not depend on the sketched topological background.

\subsubsection{Flow-polyhedron $P$}
We introduce necessary graph-theoretic terminology:
\begin{definition} Let $G_n$ be the \textit{complete digraph} with $n$ vertices $[n]=\{1,...,n\}$. Given a vertex $i\in [n]$ and a map $f: [n]^2 \rightarrow \RR$ on edges, we define the inflow and outflow of $f$ at $i$ by $\inflow_i(f)=\displaystyle \sum_j f_{ji}$ and $\outflow_i(f)=\displaystyle \sum_j f_{ij}$.  Here $f_{ij}=f(i,j)$ denotes the value of the map $f$ on the directed edge from $i$ to $j$.

A nonnegative map $f$ on edges is called a \textit{flow} if $\inflow_i(f)=\outflow_i(f)$ at all vertices. We write $W_n\subset \RR_{\geq 0}^{n^2}$ for the cone of such flows.
\end{definition}
We now define a connectedness-notion on MD-graphs:
\begin{definition} An MD-graph is \textit{connected }if for all vertices $i,j$, there is a directed path from $i$ to $j$. The graph is \textit{weakly connected} if replacing all directed edges by undirected ones turns it into a connected undirected graph.
\end{definition}
Fix $z\in M_n$. The following objects are the key ingredients in the definition of $P$:

\begin{definition} \label{disc}
We introduce the \textit{weight-function} $h_z: W_n \rightarrow \RR^\NN$ on flows as: $$(h_z(f))_i=\displaystyle \sum_{j=1}^n z^{(i)}_j  \outflow_j(f)$$ Its vanishing will mirror homological triviality of a half of the loop on the level of the representing vectors.

The cone $V(z)$ is defined as the set of flows for which $h_z$ vanishes. 
The nonzero integral vectors in $V(z)$ with connected support form the set $D(z)$ of \textit{disc-vectors}.
\end{definition}

We are now in a position to define the initially mentioned, crucially important flow-polyhedron $P(z)$ examined in our paper and the function $\kappa_z$ it determines: 

\begin{definition}
The rational, \textit{a posteriori} finite sided flow-polyhedron $P(z)$ is defined as $P(z)=\conv(D(z)+V(z))=\conv(D(z))+V(z)$. 

The \textit{sail} $S(z)$ is the boundary of this polyhedron. 

The \textit{Klein-function} $\kappa_z$ is the unique $1-$homogeneous function satisfying:  
\begin{itemize}
  \item If the ray $[v]$ passing through $v\in V$ has $[v]\cap S(z)=\emptyset$, then $\kappa_z(v)=0$.
  \item If $v$ is the closest point to $0$ in $[v]\cap S(z)$, then $\kappa(v)=1$.
\end{itemize}

\end{definition}
Roughly speaking, $\kappa_z$ is the $1$-homogeneous function whose unit set is the sail.

There is a more practical definition of $\kappa_z$ established in Lemma $3.10$ of \cite{[4]}: 
\begin{lemma} \label{sec:expressions} 
For $v\in V(z)$, an admissible expression is defined to be a representation of the form $v=\sum_j t_j d_j +v'$, where $d_j \in D(z), v' \in V(z)$ and $t_j >0$. Then $\kappa_z(v)=\sup(\sum_j t_j)$, where the supremum runs over all admissible expressions.
\end{lemma} 

\subsubsection{Calegari's formula}
If the surface $S$ rationally bounds $[\gamma]=\phi(x,y)$, the representing pair of vectors $(v_A,v_B)$ must not only lie in $V(x) \times V(y)$, but also be \textit{paired}. This property reflects that the two simple components can be glued back together and is defined as follows:
\begin{definition} For $(x,y)\in M_n\times M_n$, define the set of \textit{paired vectors} to be \footnote{We adopt the convention that $1-1=n$ here} $$Y(x,y)=\{(v_A,v_B)\in V(x)\times V(y) \ | \ \forall i,j: (v_A)_{ij}=(v_B)_{(j-1) i}\}$$  
Moreover, we define the compact set of \textit{unit-outflow vectors} as: $$Y_1(x,y)=\{(v_A,v_B)\in Y(x,y) | \forall i: \outflow_i(v_A)=\outflow_i(v_B)=1\}$$
Notice that $Y_1^n:=Y_1(x,y)$ is independent of $(x,y)\in M_n\times M_n$.
\end{definition}
We can finally reap the benefits of our endeavors and state Calegari's formula:
\begin{theorem} (Calegari)\label{theorem:calegari}
If $g=\phi(x,y)\in \ZZ^\infty \ast \ZZ^\infty$ for $x,y\in M_n \times M_n$, then:
$$\scl(g)=   \frac{1}{2} \left( n -\displaystyle \max_{(v_A,v_B)\in Y_1(x,y)} (\kappa_x(v_A) + \kappa_y(v_B)) \right)$$
\end{theorem}
\section{Estimates} \label{sec:chaptertwo}
This section is divided into three parts: we first pursue the relative approach with linear-algebraic means and then give upper and lower bounds which allow us to determine the generic value of $\scl$. 
In the final section, we prove that the determination of $\scl$ is hard in a precise sense.
\subsection{Linear Algebra of Exponents}
The first, easily proven bound relates the $\scl$ of certain words of same reduced length.
\begin{lemma} (Inequality)\label{Inequality Lemma} Let $v=\phi(r,s)$ and $w=\phi(t,u) \in \ZZ^\infty \ast \ZZ^\infty$ have equal reduced length $m=2n$, and assume that there are inclusions of $\ZZ$-modules:
 $$\langle \{r^{(i)}\} \rangle \subseteq \langle \{t^{(i)}\}\rangle \mbox{ and }\langle \{s^{(i)}\} \rangle \subseteq \langle \{u^{(i)}\} \rangle $$
 Then, we have the inequality $$ \scl(v) \leq \scl(w)$$\end{lemma}
\begin{proof}
Observe that for $z\in M_n$ and $f$ a flow, the function $h_z$ vanishes on $f$ if and only if the vector $(\outflow_j(f))_j$ is perpendicular to all vectors $\{z^{(i)}\}$. This immediately gives the inclusion $V(r) \times V(s) \supseteq V(t) \times V(u)$. With Lemma \ref{sec:expressions}, we then see that $\kappa_r+\kappa_s \geq \kappa_t+\kappa_u$, wherever defined. Since the sets over which we are maximising are both equal to $Y_1^n$, Calegari's formula gives the result. 
\end{proof}

The lemma shows that the map $\scl \circ \phi : M_n \times M_n \rightarrow \RR$ \textit{factors through spaces}, i.e.~that $\langle \{r^{(i)}\}\rangle \times \langle \{s^{(i)}\}\rangle = \langle \{t^{(i)}\}\rangle \times \langle \{u^{(i)}\}\rangle \mbox{ implies } \scl(\phi(r,s))=\scl(\phi(t,u))$
\begin{example} This property can help us in concrete cases, e.g.~we can see that:
$$\scl( a_1^2 a_2^{-3} b a_1^{-2} a_2^1 b a_1^3 a_2^1  b a_1^{-3} a_2^1 b^{-3})
=\scl( a_1^2 a_2^{-3} a_3^5 b a_1^{-2} a_2^1  a_3^{-3}   b a_1^3   a_2^1  a_3^2 b a_1^{-3}   a_2^1   a_3^{-4}  b^{-3})$$
since $(5, -3, 2, -4)=(2, -2, 3, -3) - (-3, 1, 1, 1)$.
\end{example}

The last fact can be used to prove the Compactness Lemma, which touches the ``inverse problem". Assume $w$ is a word with short reduced length which uses a large number of generators of the free abelian factors. Then our next lemma shows that we actually only need a small number of generators to realise the $\scl$ of $w$:
\begin{lemma}(Compactness) If $w=\phi(x,y)\in \ZZ^\infty \ast \ZZ^\infty$ has reduced length $2N$, its $\scl$ is already contained in the image $\scl(\ZZ^r \ast \ZZ^s)$ for all $r,s\geq N$.\end{lemma}
\begin{proof} Since $\ZZ^N$ is a free module of rank $N$ over a principal ideal domain, the two submodules generated by $\{x^{(i)}\}$ and $\{y^{(i)}\}$ respectively are free of rank at most $N$.
Therefore, we can choose generating sets $\{r_1,...,r_N\}$ and $\{s_1,...,s_N\}$.

Consider $w_0=\phi((r_1,...,r_N,0,...),(s_1,...,s_N,0,...))$. Since the map $\scl \circ \phi$ factors through spaces, we have $\scl(w)=\scl(w_0)$. We complete the proof by noticing that for all $n,m\geq N$, the word $w_0$ lies in the image of the obvious $\scl-$preserving inclusion $\ZZ^r \ast \ZZ^s \rightarrow \ZZ^\infty \ast \ZZ^\infty$.
\end{proof}

\subsection{Generic Value of scl}
In order to determine the generic behaviour of $\scl$, we will need to bound it from above and below. 

The first theorem of this section provides a lower bound on $\scl$ in terms of a subset-sum type problem determined by the exponents of our word. The immediate relation between the outer form of the word and the type of the bound makes it particularly powerful, as can be seen in the remainder of this section.

\begin{theorem}(Lower Bound) \label{Lower Bound Theorem}
Let $w=\phi(x,y)\in \ZZ^\infty \ast \ZZ^\infty$ have reduced length $2n$. 

Fix $p,q \in \NN$, and assume that the following two implications hold:

If $(\lambda_j)_j \in \NN^n\backslash\{0\}$ is a vector with $\sum_j \lambda_j x^{(i)}_j =0$ for all $i$, then $\sum_j \lambda_j \geq p$.

If $(\mu_j)_j \in \NN^n\backslash\{0\}$ is a vector with $\sum_j \mu_j y^{(i)}_j =0$ for all $i$, then $\sum_j \mu_j \geq q$.

In this case, we have the inequality $\scl(w) \geq \frac{n}{2}(1-\frac{1}{p}-\frac{1}{q})$
\end{theorem}

\begin{proof} The main idea behind this proof is to use the given implications to show that the disc-vectors have to contain a ``large amount of flow".

By compactness, there is a vector $(v_A,v_B)\in Y_1(x,y)$ maximising $\kappa_x+\kappa_y$. Let $v_A=\sum_k t_k d_k + v'$ be any admissible expression with $t_k >0$, $d_k \in D(x), v' \in V(x)$. 

\begin{claim} There is an inequality $\kappa_x(v_A)\leq \frac{n}{p}$ \end{claim}
\begin{proof}[Proof of claim]
The proof proceeds in three steps:

\textbf{Step (1):}  By definition, the function $h_x$ vanishes on disc-vectors. This means that $(h_x(d_k))_i=\sum_j x^{(i)}_j \outflow_j(d_k)=0$ for all $i,k$.
Using the first implication of the theorem, we conclude that $\sum_j \outflow_j(d_k) \geq p$ for all $k$.

\textbf{Step (2):}  Since $(v_A,v_B)$ lies in the set $Y_1(x,y)$, we have $\outflow_j(v_A)=1$ for all vertices $j$.
Using our admissible expression from above, we conclude that $\sum_k t_k \outflow_j(d_k) \leq 1$ for all $j$.

\textbf{Step (3):} We swap the order of summation:
$$p\sum_k t_k \leq \sum_k \sum_j t_k \outflow_j(d_k) = \sum_j \sum_k t_k \outflow_j(d_k) \leq n$$
By Lemma \ref{sec:expressions}, we then have $\kappa_x(v_A)=\sup( \sum_k t_k)$, where the supremum runs over all admissible expressions. This implies the claim. \end{proof}

Similarly, we prove $\kappa_y(v_B) \leq \frac{n}{q}$. The result follows by Calegari's formula \ref{theorem:calegari}.\end{proof}
A major open problem in the theory of $\scl$ is to prove the conjecture that $\scl(F_2)$ contains every rational number $q\geq 1$. By Lemma $3.18$ in \cite{[5]}, if $s,t\in \scl(F_2)$, then also $s+t+\frac{1}{2}\in F_2$. Writing $q=\frac{1}{2}+(q-1)+\frac{1}{2}$, we see that the above conjecture is equivalent to:
\begin{conjecture}(Interval) The image $\scl(F_2)$ contains every rational number in the interval $[1,2]$.
\end{conjecture} 

The next corollary of the Lower Bound Theorem gives an indication where to look for these $\scl$-values: Either they come from short words or from words which walk along the same subloop twice in opposite directions (like in figure \ref{fig:CT}).

\begin{corollary}Let $w=\phi(x,y)$ be a word of reduced length $m=2n$ with $\scl(w)\in [1,2]$. Then either $m\leq 24$ or there are distinct indices $j_1,j_2$ in $\{1,...,n\}$ such that we have one of the two following identities of vectors:
$$(x^{(i)}_{j_1})_i = -(x^{(i)}_{j_2})_i \mbox{  and   }  (y^{(i)}_{j_1})_i = -(y^{(i)}_{j_2})_i$$
\end{corollary}
\begin{proof}
Suppose that the second possible conclusion does not hold. Since all vectors $(x^{(i)}_j)_i, (y^{(i)}_j)_i$ are nonzero, we can apply the Lower Bound Theorem \ref{Lower Bound Theorem} with the values $p=q=3$ to obtain $2\geq \scl(w)\geq \frac{n}{6}$ \end{proof}
After having bounded $\scl$ from below, we will now proceed to give an upper bound purely in terms of the reduced wordlength. This bound is sharp in the limit and will be the second key ingredient in the determination of the generic behaviour of $\scl$ in Theorem \ref{Generic Word Theorem}.

\begin{theorem}(Upper Bound)\label{Upper Bound Theorem} Write $C(m)$ for the supremum of the $\scl$ of words in $\ZZ^\infty\ast \ZZ^\infty$ of reduced length $m=2n>4$. Then this supremum is attained and satisfies
$$C(m)\leq   \left\{\begin{array}{cl} \frac{n}{2}-1, & \mbox{if $n$ odd}\\ 
\frac{n}{2}-\frac{(n-1)!-1}{n(n-2)!-2}, & \mbox{if $n$ even} \end{array}\right.$$
Moreover, given $\epsilon>0$, we have for $m$ sufficiently large: $\frac{m}{4}-1\leq C(m) \leq \frac{m}{4}-1+\epsilon$
\end{theorem}
\begin{proof}
For a given reduced length $2n$, we will intersect all cones $V(x)$ to obtain a word $w_n$ whose $\scl$ is maximal for this length.
More precisely, define $z\in M_n$ by:
$$z^{(1)}=	\begin{pmatrix} 1&-1&...&0 \end{pmatrix},\ ...\ ,
z^{(n-1)}=	\begin{pmatrix} 1&0&...&-1 \end{pmatrix},
z^{(n)}=z^{(n+1)}=\ ...\ =0$$
Then the universally bounding word is given by $w_n=\phi(z,z)$.

For all $x,y \in M_n$, we have: $\langle \{x^{(i)}\} \rangle, \langle \{y^{(i)}\} \rangle \subseteq \langle \{z^{(i)}\}\rangle = \{f\in \ZZ^n|\sum_j f_j = 0\}$.
Therefore, the Inequality Lemma \ref{Inequality Lemma} gives that $\scl(w)\leq \scl(w_n)$ for all words $w$ of reduced length $2n$.

We will now bound $\scl(w_n)$. Notice that $V(z)=\displaystyle \bigcap_{x\in M_n} V(x)$ is exactly the set of flows $f$ on $G_n$ for which $\outflow_i(f)$ is equal at all vertices $i$. In particular, all Hamiltonian cycles on $\{1,...,n\}$ are disc-vectors.
We split cases:

\textbf{n odd:} Let $v_A$ be the flow corresponding to the cycle $1, 3, 5 ,..., 2, 4, 6 ,... , 1$. We are then forced by the pairing condition to define $v_B$ to be the flow corresponding to 
$1, n, n-1,...,2,1$. Clearly, we have $(v_A, v_B)\in Y_1^n$, and since both vectors $v_A, v_B$ are disc-vectors, we also conclude that $\kappa_z(v_A),\kappa_z(v_B)\geq 1$. Calegari's formula \ref{theorem:calegari} and the Lower Bound Theorem $\ref{Lower Bound Theorem}$ then imply: $$\scl(w_n) = \frac{n}{2} -1$$

\textbf{n even:}  Let $v$ be the vector obtained by adding the $n$ possible rotations of:
\begin{figure}[ht!]
\centering
\labellist
\pinlabel $(n-2)!-1$ at -15 170
\pinlabel $(n-2)!-1$ at 162 148
\pinlabel $(n-2)!$ at 185 105
\pinlabel $(n-2)!$ at 175 40
\pinlabel $(n-2)!$ at 123 2
\pinlabel $(n-2)!$ at 52 2
\pinlabel $(n-2)!$ at -8 40
\pinlabel $(n-2)!$ at -17 91
\endlabellist
\includegraphics[scale=0.6]{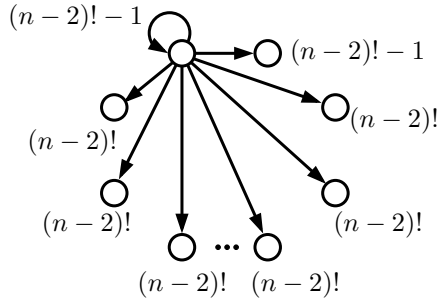}
\caption{Vector in $\RR^{n^2}$; the vertices $\{1,...,n\}$ are arranged clockwise. The number attached to a vertex gives the value on the edge going into this vertex.}
\label{fig:rotation}
\end{figure}

Then $v$ is a flow in $V(z)$ with $\outflow_i(v)=n (n-2)!-2$. Notice that $v$ contains the flow which is obtained by summing up all possible Hamiltonian cycles except for $1,2,...,n,1$.
Since there are $(n-1)!-1$ of them, Lemma \ref{sec:expressions} implies that $\kappa_z(v)\geq (n-1)!-1$. We observe that $(v,v)\in Y(z,z)$ is a paired vector, so
$(v_A,v_B):=(\frac{v}{n (n-2)!-2},\frac{v}{n (n-2)!-2})\in Y_1(z,z)$ lies in the set we need to maximise over. 
We now use Calegari's formula \ref{theorem:calegari}, the $1-$homogeneity of $\kappa_z$ and the Lower Bound Theorem $\ref{Lower Bound Theorem}$ to conclude that $\frac{n}{2}-1 \leq \scl(w_n)\leq \frac{n}{2} - \frac{(n-1)!-1}{n(n-2)!-2}$.

We finally observe that $C(n)=\scl(w_n)$, which completes the proof.\end{proof}

\begin{remark}
For $n=4$, we have $\scl(w_4) = \frac{7}{6}=2-\frac{(4-1)!-1}{4(4-2)!-2}$ (computed with the implementation of Calegari's algorithm by A. Walker \cite{[7]}) and hence the estimate for even $n$ is sharp in at least one case.
A precise computation of $C(n)=\scl(w_n)$ for general even $n$ remains an open problem.\end{remark}

Finally, we now have all the necessary tools at our disposal to prove the main theorem of this section and determine the generic behaviour of $\scl$. Recall the definition \ref{def:generic} of \textit{generic properties}.

\begin{theorem}(Generic Words) \label{Generic Word Theorem}
Given any $\epsilon>0$, we can choose $N$ such that for all $m\geq N$, words $w$ of reduced length $m$ generically satisfy
$$ \scl(w)\in \left[\frac{m}{4}-1,\frac{m}{4}-1+\epsilon\right]$$
\end{theorem}
\begin{proof}
Let $\epsilon>0$. By the Upper Bound Theorem \ref{Upper Bound Theorem}, we can pick $N$ such that for all $2n\geq N$, words $w$ of length $2n$ have $\scl(w)\leq\frac{n}{2}-1+\epsilon$. Fix $n$ and set $V=\{z\in \ZZ^n| \sum_j z_j =0\}$. For each of the finitely many vectors $\lambda \in \NN^n$ with $0<\sum_j \lambda_j < n$, define the space $W_\lambda$ by:
$$W_\lambda = \left\{ z\in V \bigg| \sum_j \lambda_j z_j=0\right\} = \ker \begin{pmatrix} 1&1&...&1 \\ \lambda_1 & \lambda_2 & ... & \lambda_n  \end{pmatrix}$$
Since the rows of this matrix are linearly independent, the submodule $W_\lambda$ has rank at most $n-2$.

Let $w=\phi(x,y)$ be any word of reduced length $2n$. If not all $x^{(i)}$ and not all $y^{(i)}$ lie in 
$\displaystyle \bigcup_\lambda W_\lambda$, then the Lower Bound Theorem \ref{Lower Bound Theorem} helps us to prove that $\scl(w)\geq \frac{n}{2}-1$. The theorem follows.
\end{proof}
\begin{remark} Calegari's algorithm also applies to free products of $k>2$ free abelian groups, and it is natural to ask which results carry over to this more general setting. Let $w$ be a word with $n_i$ nontrivial loop segments in the $i^{th}$ group of our product.

Again, we have $k$ flow-polyhedra $P_1,...,P_k$, where $P_i$ is constructed from the exponents of letters in the $i^{th}$ group by the same procedure as before. The definition of the associated functions $\kappa_1,...,\kappa_n$ also carries over. To compute $\scl$, we have to maximize their sum over a compact subset $Y_1$ of the product of all $P_i$. This sum is obtained by first restricting to unit-outflow vectors, and then imposing a ``gluing condition", whose exact form is more complicated than before (see \cite{[4]}).

The Lower Bound Theorem generalises: Assume we are given a word $w$ of length $m$ whose exponents in the $i^{th}$ free factor are $^i\mathbf{x}_{1}$,$...$,$^i\mathbf{x}_{n_i}$, and such that we need to sum at least $p_i$ of these exponents (with repetitions) to obtain zero. By the same estimates on $\kappa_i$ as before, we obtain: $$\scl(w)\geq \frac{m}{4}-\frac{1}{2}\sum_i \frac{n_i}{p_i}$$

However, the proof of the Upper Bound Theorem breaks down in this more general setting. Our strategy for $k=2$ was to give every edge roughly the same weight so that the pairing condition holds independently of the precise form of the word. The vector obtained in this way had equal outflow at all vertices since both graphs had the same cardinality, so it was possible to rescale and obtain a required unit-outflow vector. 

For $k>2$, the cardinalities of the involved graphs are usually very different, and therefore this vector cannot be rescaled anymore to have unit outflow everywhere.

Therefore the proof of the Generic Word Theorem does not generalise. We can still deduce from the generalised Lower Bound Theorem that words $w$ with $n_i$ nontrivial loop segments in the $i^{th}$ group generically have $\scl(w)\geq \frac{m}{4}-\frac{k}{2}$.
\end{remark}
\subsection{Complexity of Computing scl}
We will give a lower bound on the \textit{algorithm independent} complexity of computing $\scl$ of efficiently encoded words in $F_2$:
\begin{theorem}(Complexity)\label{complexity} Unless $\P=\NP$, the $\scl$ of words $\phi(x,y)\in F_2$ cannot be computed in polynomial time in the input size of the vector $(x,y)$.\end{theorem}

After briefly reviewing basic complexity-theoretic notation and previous results on $\scl$, the main aim of this section is to prove this Complexity Theorem with the techniques from the preceding section. 
\begin{remark}
An algorithm is \textit{polynomial} if it runs in time polynomially bounded in terms of the size of its input.

A decision (or promise) problem is said to be polynomially solvable if there is a \textit{polynomial} algorithm that solves it. We say the problem lies in the class \P.

A decision (or promise) problem lies in the complexity class \NP{} if a solution to the problem can be checked in polynomial time in the input size. The problem is said to be \NP-complete if it lies in \NP{} and if no other problem in this class is harder in a precise sense.

The input size of a computational task is the number of bits required to encode the input binarily. For example, a vector $(x_1,...,x_n)$ of natural numbers requires roughly $\sum_i \log_2(x_i)$ bits.
\end{remark}

\begin{remark}
There is a simple measure of length in the free group $F_2=\ZZ \ast \ZZ$ other than the reduced length from \ref{def:reduced}: the \textit{classical length} of a word is the number of letters in a shortest representation. This notion does not naturally extend to groups $\ZZ^r\ast \ZZ^s$ for $r>1$ or $s>1$ as it depends on a choice of bases for the factors. The alternative $\scl$-algorithm presented in section $4.1.7$ of \cite{[3]} shows that $\scl$ can be computed in polynomial time in the \textit{classical length}. More precisely, if we encode words as binary strings with pairs of entries representing generators / their inverses, then we can compute $\scl$ in polynomial time in the size of this (very large) input.
 
However, it is artificial and inefficient to waive the use of exponents and use unary coding, so for example to write $aaaaababa^{-1}a^{-1}a^{-1}a^{-1}a^{-1}a^{-1}b^{-1}b^{-1}$ for the word $a^5b ab a^{-6} b^{-2}$. But ``using exponents" is just the colloquial term for encoding our words via the map $\phi$, adapted to $F_2$. We will therefore consider the problem of computing $\scl(\phi(x,y))$ for $x,y\in (\ZZ\backslash \{0\})^n$ as input. Our proof will show that unless $\P=\NP$, there is no algorithm which solves this problem in polynomial time in the size of the vector $(x,y)$. Notice that this is strictly stronger than just saying that unless $\P=\NP$, $\scl$ cannot be computed in polynomial time in the reduced length.
\end{remark}

\subsubsection{Complexity-theoretic Notation} 
The following classical problem is known to be $\NP-$complete and will be the starting point of our complexity-theoretic analysis:
\begin{problem}
\SUBSETT{}  Given $(r_1,...,r_n)\in \ZZ^n$. Is there a nonzero vector $(\lambda_1,...,\lambda_n)\in \{0,1\}^n$ with $\sum_{j}\lambda_j r_j = 0$? \end{problem}

This problem can be modified in several different ways, and we will now give names to the variations we need.

We first restrict ourselves to cases where the whole input is known to sum up to zero and ask for proper nonempty subsets whose sum is zero: 
\begin{problem}\SUBSETTP{} Given $(r_1,...,r_n)\in \ZZ^n$ with $\sum_j r_j = 0$. Is there a vector $(\lambda_1,...,\lambda_n)\in \{0,1\}^n$ with $0<\sum_j \lambda_j <n$ and $\sum_{j}\lambda_j r_j = 0$? \end{problem}

We vary this problem and allow the repeated use of individual $r_j$'s, but keep the total number of employed $r_j$'s bounded: 
\begin{problem} \VARSUBSETTP{} Given $(r_1,...,r_n)\in \ZZ^n$ with $\sum_j r_j = 0$.
Is there a vector $(\lambda_1,...,\lambda_n)\in \NN^n\backslash \{0\}$ with $0<\sum_j \lambda_j < n$ and $\sum_j \lambda_j r_j =0$? \end{problem}

We combine two problems and obtain the following promise problem:
\begin{problem} \MIXEDSUBSETTP{} \label{mixedsubset}  Given $(r_1,...,r_n)\in \ZZ^n$ with $\sum_j r_j = 0$ such that \SUBSETP \
holds iff \VARSUBSETP{} does. Are they satisfied?\end{problem}

Finally, we define the decision problem which will serve as the key gadget in the proof of the Complexity Theorem \ref{complexity}:
\begin{problem} \SMALLSCLT{} Given a vector $x\in (\ZZ \backslash\{0\})^n$ with $\sum_j x_j = 0$. Define $y=(1,..,1,-(n-1))\in M_n$.
Is it true that $\scl(\phi(x,y))<\frac{n}{2}-1$?
\end{problem}

\subsubsection{Proof of Complexity Theorem}
\begin{proof} Our aim is to polynomially reduce \SUBSET{} to \SMALLSCL. This proves the Complexity Theorem: if we could compute $\scl$ for words encoded with $\phi$ in polynomial time in the input size of the vectors, then it would be possible to answer \SMALLSCL{} and thus also \SUBSET{} in polynomial time.  This would then imply $\P=\NP$.

Our reduction passes through \MIXEDSUBSETP: A combinatorial argument due to F. Manners reduces \SUBSET{} to \MIXEDSUBSETP. A proof is attached in the Appendix.

Hence we are left with reducing \MIXEDSUBSETP{} to \SMALLSCL. This reduction relies on the Lower Bound Theorem \ref{Lower Bound Theorem} and the following following relation between $\scl$ and the \SUBSETP{} problem:
\begin{lemma}\label{subsetlemma} Let $x\in (\ZZ\backslash\{0\})^n$ be a vector with $\sum_j x_j = 0$. Assume there is a nonempty set $J\subset \{1,...,n\}$ of size $M\leq \frac{n}{2}$ with $\sum_{j\in J} x_j=0$ and moreover that neither $J$ nor $J^c$ are of the form \footnote{In this Lemma, we work with addition modulo $n$, picking representatives in $\{1,...,n\}$.} $\{k, k+1\}$.Then we have: $$\scl(a^{x_1}b...a^{x_{n-1}}ba^{x_n}b^{-(n-1)})\leq \frac{n}{2}-1-\frac{1}{(n-M-1)!} < \frac{n}{2}-1$$
\end{lemma}
\begin{proof} Set $y=(1,...,1,-(n-1))$. We will first give a unit-outflow vector $v_A\in V(x)$ with $\kappa_x(v_A)\geq 2$, and then find $v_B\in V(y)$ such that $(v_A,v_B)$ is paired and $\kappa_y(v_B)$ is large enough.

A pair of flow-vectors $(c,d)$ on $G_n$, the complete digraph on vertices $\{1,...,n\}$, shall be called \textit{$J-$pair} if $c$ represents a Hamiltonian cycle on $J$ and $d$ represents one such cycle on $J^c$. Observe that we can pick $N=(n-M-1)!$ such $J$-pairs $(c_1,d_1),...,(c_N,d_N)$ so that their sum $v=\sum_j (c_j+d_j)$ has nonzero flow on all edges inside $J$ and all edges inside $J^c$. The assumptions of this theorem tell us that the vector $v\in V(x)$ decomposes into at least $2N$ disc-vectors since the relevant sums vanish. Using Lemma \ref{sec:expressions}, we obtain that vector $v_A = \frac{v}{N}$ has $\kappa_x(v_A)\geq 2$. Define $$(v_B)_{ij}=(v_A)_{j (i+1)}$$ One checks easily that $v_B$ is a flow with unit outflow everywhere. Therefore, we know that $v_B\in V(y)$ and $(v_A,v_B)\in Y_1(x,y)$.

We are left with proving that $\kappa_y(v_B)\geq \frac{1}{N}$. It is enough to show that $v_B$ has connected support, since then $Nv_B$ is a disc-vector and we therefore have $\kappa_y(Nv_B)\geq 1$. A digraph is called \textit{weakly connected} if replacing all directed edges by undirected ones turns it into a connected graph. By Proposition \ref{weaklyconn} proven below, it is enough to show that the support of the flow $v_B$ is weakly connected.

If $I=\{p,..,(p+s)\}\subset J$ and $(p-1),(p+s+1) \notin J$, we say $I$ is an \textit{interval in $J$}. We have an obvious analogous definition for \textit{intervals in $J^c$}. Then $\{1,...,n\}$ decomposes into intervals $I_1,...,I_m$, where $I_j$ is in $J$ for $j$ odd and in $J^c$ for $j$ even.

Our goal is to show that given an interval $I_k=\{p,...,p+s\}$ in $J$, all points in $I_k$ lie in the same weakly connected component as $p-1$. Split cases:

\textbf{Case (1):} $s=0$. In this case, $(v_B)_{p (p-1)} = (v_A)_{(p-1) (p+1)}>0$ and so $p$ and $p-1$ are weakly connected.

\textbf{Case (2):} $s\geq 1$. By the form of $J$, we can pick $q\neq p,(p+1)$ with $q\in J$.
One checks that the edges indicated below have positive flow for $v_B$:
\begin{figure}[ht!]
\centering
\labellist
\pinlabel {\Huge $J$} at 222 218
\pinlabel {\Huge $J^c$} at 550 218
\pinlabel {\small $p$} at 115 285
\pinlabel {\small $p\! +\!\!1$} at 174 314
\pinlabel {\small $p\!+\!\!s$} at 278 314
\pinlabel {\small  $p\! -\!\!1$} at -20 349
\pinlabel {\small  $p\!+\!\!s\!+\!\!1$} at 425 405
\pinlabel {\small $q$} at 140 125
\endlabellist
\includegraphics[scale=0.25]{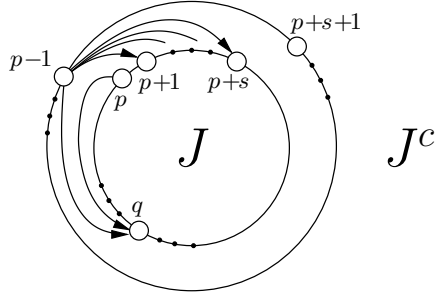}
\caption{Flow-vector in $\RR^{n^2}$; the vertices $\{1,...,n\}$ are arranged clockwise}
\label{fig:figure6}
\end{figure}

Exactly the same argument holds for intervals in $J^c$. Combining these two claims, we see that $I_k$ and $I_{k-1}$ lie in the same weakly connected component for all $k$. We go once around the circle to conclude that the weak support of $v_B$ is connected.
\end{proof}
We are now ready to reduce \MIXEDSUBSETP{} to \SMALLSCL:
\begin{lemma}
If we can solve \SMALLSCL{} in polynomial time, then we can also solve \MIXEDSUBSETP{} in polynomial time.
\end{lemma}
\begin{proof}
Given a problem instance $(r_1,...,r_n)\in \ZZ^n$, we can check in polynomial time if there is a $j$ with $r_j=0$ or $r_j+r_{j+1}=0$. 

If this is not the case, we compute 
$$s=\scl(a^{r_1}b...a^{r_{n-1}}ba^{r_n} b^{-(n-1)})$$ 

If $s<\frac{n}{2}-1$, then by the Lower Bound Theorem \ref{Lower Bound Theorem}, we can deduce that the problem \VARSUBSETP$(r_1,...,r_n)$ is true, hence so is \MIXEDSUBSETP$(r_1,...,r_n)$.

If $s\geq\frac{n}{2}$, then we use Lemma \ref{subsetlemma} to conclude that \SUBSET'$(r_1,...,r_n)$ and hence \MIXEDSUBSETP$(r_1,...,r_n)$ are both false.
\end{proof}
This concludes the proof of the Complexity Theorem \ref{complexity}. \end{proof}

\section{Polyhedra} \label{sec:chapterthree}
\subsection{Non-Classifiability Theorem}\label{classi}
Whereas the extremal rays of the $\scl-$poly-hedra $P(z)=\conv(D(z)+V(z))$ have been classified in a satisfactory manner (see Figure \ref{fig:FIGURE2}), such a description could not be found for their extremal points. Our main theorem, whose proof will occupy most of this section, gives a reason for this:

\begin{theorem}(Non-classifiability) \label{classiclassi} Let $G$ be a connected MD-graph with $M$ vertices and $E>0$ edges. Then there is an (alternating) word $w\in \ZZ \ast \ZZ$ of length $m=4(M+3E^{3ME+1})$ and an extremal point $f$ of a flow-polyhedron associated to $w$ such that $G$ is the abstract graph underlying $f$.
\end{theorem}

We start this section with a brief treatment of extremal points, and then present a more general version of the Non-Classifiability Theorem, deducing the specific case from there. In the third part, we then prove the generalised theorem.
\subsubsection{Extremal Points of Polyhedra}
The next definition is central:
\begin{definition} Let $S\subset \RR^n$ be a subset. A vector $x\in S$ is called an \textit{extremal point of S} if $x$ cannot be written as a nontrivial convex combination of other vectors in $S$, i.e.~if $d=\displaystyle \sum_{i=1}^n \lambda_i v_i$ for $\lambda_i \in \RR_{>0}, \displaystyle \sum_j \lambda_j =1, v_j \in S$ implies $v_j = d$ for all $j$.
\end{definition}
Let $E\subset \ZZ^n$ be a set of integral vectors and $P=\conv(E)$ the polyhedron defined by its convex hull. The following is a very useful criterion for deciding whether a given point in $E$ is extremal:

\begin{lemma}(Extremality Criterion)\label{extremality}
A vector $d\in P$ is an extremal point if and only if $d\in E$ and:
$$ \mbox{For all }N \in \NN \mbox{ and } d_1,...,d_N\in E: \Big( Nd = d_1+....+d_N \mbox{ implies } d_1 = ... = d_N = d \Big) $$
\end{lemma}
\begin{proof}
Clearly, an extremal point $d$ must lie in $E$, else it would need to be a nontrivial convex combination of vectors in $E$. If $(Nd = d_1+....+d_N)$ is an expression as above, dividing by $N$ gives a convex combination. We conclude $d_i=d$ for all $i$.

Conversely, assume $d\in E$ is a vector for which the above implication holds. An easy computation shows that $d$ is extremal if and only if it cannot be written as a nontrivial convex combination of vectors in $E$. Suppose $d=\sum_j \lambda_j d_j$ is a convex combination with $\lambda_j>0$ and $d_j\in E$ for all $j$.

If $d_j=d$ for all $i$, we are done. 

If not, we can bring all $\lambda_j d_j$ with $d_j=d$ to the left hand side, rescale, and hence obtain a representation of $d$ as a convex combination of vectors that are all different from $d$. Therefore, we may assume $d_j\neq d$ for all $j$. 
As the rational vector $d$ lies in the real convex span of the rational points $d_1,...,d_N\in \ZZ^n$, we know by Lemma $4$ in \cite{[2]} that $d$ also lies in their rational convex span. 
Thus we can find $\mu_j = \frac{p_j}{q_j}$ with $p_j \in \NN_0$, $q_j \in \NN $ such that $d=\sum \mu_j d_j$. Then, for $N = \displaystyle \prod_j q_j$, we have: $$Nd= \displaystyle \sum_i p_i \left(\displaystyle \prod_{j\neq i} q_j\right) d_i$$
This is a sum of $\displaystyle \sum_i p_i \left(\displaystyle \prod_{j\neq i} q_j\right) = N \displaystyle \sum_i \frac{p_i}{q_i} = N$ vectors in $E$, so by the implication in the theorem, we have for all $i$ with $p_i \neq 0$ that $d = d_i$. This is a contradiction.
\end{proof}

\subsubsection{Generalised Non-Classifiability Theorem} We need new definitions to state the general version of the theorem.
Recall Definition \ref{MDgraphs} of MD-graphs and of the function $\abst$. We can associate polyhedra to edge-weights on graphs:
\begin{definition} Let $G$ be an MD-graph with edge-weights $w$. Consider the set $E_w$ of nonzero (nonnegative) integral flows $f$ for which $\sum_e f(e) w(e)=0$.  Define \textit{polyhedron associated to $w$} by $Q_w=\conv(E_w)$.
\end{definition}

As an interpretation, we can think of $G$ as a country with cities and connecting streets, where transferring a food-unit over route $e$ costs $w(e)$ money-units. Then $E_w$ are the nonzero integral food-flows for which our selfless state does not earn or loose money with its road toll/subsidy system. 

We extend this definition to vertex-weights:
\begin{definition}
Let $v$ be a vertex-weight on $G$. Define an edge-weight $w_v$ by giving an edge from $a$ to $b$ the weight $v(a)$. Set $E_v=E_{w_v}$ and $Q_v=Q_{w_v}$.
\end{definition}

In our above motivation, this corresponds to the state adjusting its tolls to how desirable it is to leave certain cities. Notice that for $x\in (\ZZ\backslash\{0\})^n$ a vertex-weight on the complete digraph $G_n$, the set $E_v$ contains precisely the nonzero integral vectors in the set $V(x)$ defined in Definition \ref{disc}.

Recall that an MD-graph is \textit{reflexive} if its weakly connected components agree with its connected components. In this situation, we can prove:

\begin{theorem}(Generalised Non-Classifiability) \label{generalclassi} Let $G$ be a reflexive abstract MD-graph with $M$ vertices and $E>0$ edges. Let $x\in \ZZ^n$ be a vertex-weight on $G_n$ containing at least $M+3E^{3ME+1}$ entries equal to $1$ and at least $M+3E^{3ME+1}$ entries equal to $-1$. Then, there is an extremal point $f$ of the polyhedron $Q_x$ whose abstract graph is $G$.
\end{theorem}

We deduce the specific Non-Classifiability Theorem from this general case:
\begin{proof}[Proof of \ref{classi} from \ref{generalclassi}]
Let $G$ be a connected, so in particular reflexive, MD-graph with $M$ vertices and $E>0$ edges. Consider the word $w=\phi(x,x)$, where $$x=
\left( \underbrace{1, ... ,1,}_{N \mbox{ times}} \underbrace{-1, ... ,-1}_{N \mbox{ times}}\right)$$
for $N=M+3E^{3ME+1}$. By the general version of the theorem, the polyhedron $Q_x$ contains an extremal point $f$ for $Q_x$ whose underlying abstract graph is $G$.

\begin{claim}
With the notation introduced in Definition \ref{disc}, we have $$P(x)=\conv(D(x)+V(x))\subset Q_x$$
\end{claim}
\begin{proof}[Proof of claim] Given $d+v\in D(x)+V(x)$, we can use Lemma $4.11$ in \cite{[4]} to express $v$ as $v=\displaystyle \sum_j \lambda_j v_j$ as a nonnegative linear combination of integral representatives $v_i$ of the extremal rays of $V$. Then for some natural $N> \sum _i\lambda_i$, we have:
 $$d+v= \sum_j \frac{\lambda_j}{N} \left( d + N v_j \right)+ \left(1- \frac{\sum_j \lambda_j}{N}\right)\cdot (d+0)$$
This implies that $d+v \in Q_x$. The claim follows as $Q_x$ is convex.\end{proof}
Hence since $f$ is an extremal point in $Q_x$ and lies in $P(x)$, it follows immediately that $d$ is also an extremal point for $P(x)$ \end{proof} 

\begin{remark} Theorem \ref{classiclassi} is sharp in the sense that only connected graphs can occur as abstract graphs underlying extremal points of $P(x)$. Indeed, all extremal points are disc-vectors as if $v=d+e$ is extremal, with $d\in D(x)$ and $e\in V(x)$, then, by considering the expression $v=\frac{1}{2}d+\frac{1}{2}(d+2e)$, we can see that $e=0$.
\end{remark}

We need a central definition before we can start the proof of the generalised theorem. We have seen in \ref{MDgraphs} that to every MD-graph $G$, we can associate an equivalent abstract graph, denoted by $\abst(G)$. The next definition describes how this construction extends to flows and weights.
\begin{definition} \label{abst2} Given a flow $g$ on a graph $G$ with support $S$. The abstract graph $\abst(g):=\abst(S)$ is naturally equipped with an induced flow $f$. 

If moreover $G$ is equipped with an edge weight $u:E(G)\rightarrow \RR$, then we construct the induced weight $w$ on $\abst(f)$ as follows: We can pass from a (finite) graph to its abstraction in finitely many steps by successively joining pairs of edges. Whenever we merge two edges $e_1,e_2$, we give the new arising edge $e_{12}$ the weight $u(e_1)+u(e_2)$. This yields a well-defined weight $w$ on the graph $\abst(f)$.  

Write $\abst(f,w)$ for the graph $\abst(f)$ with induced flow and edge-weight.
\end{definition}

\subsection{Proof of Generalised Non-Classifiability Theorem}
\begin{proof}
This rather long proof involves multiple steps. Before we fill out the details, we will give a rough sketch of how we turn a graph into an extremal point underlied by this graph. The deep reason which allows us to get such a strong control over extremal points via the Extremality Criterion \ref{extremality} is that $N$ positive integers summing up to $N$ all have to be equal to $1$.
Let $G$ be a reflexive abstract MD-graph with $M$ vertices and $E>0$ edges, together with a vertex-weight $x\in \ZZ^n$ as in the theorem. We proceed in three steps: 

\textbf{Step (1):} We find an integral flow $f$ on the graph $G$, nonzero on all edges, such that there exists an edge $e$ (drawn with a dotted line) with flow-value $1$. Moreover, $f$ satisfies $f(e')\leq E^M$ on all edges $e'$.

\textbf{Step (2):} We define an integral edge-weight $w$ on $G$, which is negative on $e$ and positive on all other edges.
The number-theoretic properties of $w$ are chosen to allow an application of the Extremality Criterion \ref{extremality}. More precisely, we will show that the flow $f$ is an extremal point of the polyhedron $Q_w$. 

\textbf{Step (3):} We implement the edge-weighted flowed graph $(G,f,w)$: that means we find an integral flow $g\in V(x)$ on $G_n$ such that $\supp(g)=G$, the induced flow of $g$ is $f$, and the edge-weight on $G$ induced from the weight $w_x$ on $G_n$ agrees with the weight $w$ from Step (2). It then follows easily that $g$ is the required extremal point of $Q_x$.

The following picture describes this construction in a simple example (notice that not all edges are drawn in the $4^{th}$ part.)
\begin{figure}[htbp]
	\centering
	\labellist 
\pinlabel {\tiny D} at 21 253.7
\pinlabel {\tiny D} at 432 253.6
\pinlabel {\tiny D} at 836 252.5
\pinlabel {\tiny A} at 177 148
\pinlabel {\tiny A} at 588 147.9
\pinlabel {\tiny A} at 992 146.8
\pinlabel {\tiny B} at 135 241.55
\pinlabel {\tiny B} at 545.7 241.52
\pinlabel {\tiny B} at 950 240.3
\pinlabel {\tiny C} at 217 326.73
\pinlabel {\tiny C} at 628 326.6
\pinlabel {\tiny C} at 1032 325.0
\pinlabel {\tiny A} at 1369 357.7
\pinlabel {\tiny B} at 1368.1 312.9
\pinlabel {\tiny C} at 1366.5 269.6
\pinlabel {\tiny D} at 1367.5 226.5
\pinlabel {Step 1} at 338 315
\pinlabel {Step 2} at 749 315
\pinlabel {Step 3} at 1183 315
\pinlabel {$L_i$} at 1368 410
\pinlabel {$R_i$} at 1575 410
\pinlabel {\small 3} at 483 373
\pinlabel {\small 3} at 894 376
\pinlabel {\tiny $w_5$} at 911 335
\pinlabel {\small 1} at 472 143
\pinlabel {\small 1} at 883 135
\pinlabel {\tiny $w_2$} at 910 170
\pinlabel {\small 2} at 484 258
\pinlabel {\small 2} at 895 258
\pinlabel {\tiny $w_4$} at 895 208
\pinlabel {\small 1} at 600 270
\pinlabel {\tiny $w_1$} at 1037 237
\pinlabel {\tiny $w_3$} at 1030 202
\pinlabel {\small 1} at 1011 270
\pinlabel {\small 1} at 600 191
\pinlabel {\small 1} at 999 201
\pinlabel {\tiny $w_6$} at 960 192
\pinlabel {\small 2} at 665 185
\pinlabel {\small 2} at 1076 184
\pinlabel {\small 1} at 1458 362
\pinlabel {\small 1} at 1287 268
\endlabellist
		\includegraphics[width=0.90\textwidth]{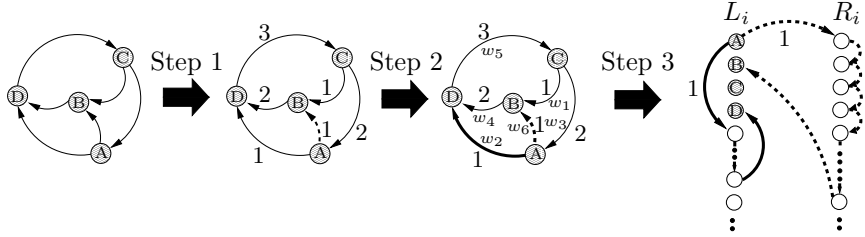}
	\caption{Finding extremal point with given abstract graph. Here $w_1=10^6\!+\!1,$\ \ \ $w_2=10^6\!+\!10, w_3=10^6\!+\!10^2, w_4=10^6\!+\!10^3, w_5= 10^6\!+\!10^4, w_6=-9,032,211$}
	\label{fig:FIGURE9}
\end{figure}\\
We now provide the details:
\subsubsection*{Step (1):}
To find the required flow, we need two lemmata. The first one characterises which graphs can appear as supports of flows.
\begin{proposition}\label{weaklyconn}
Let $G$ be an MD-graph with $M$ vertices and $E$ edges. Then $G$ admits a flow $f$ which is positive on all edges if and only if it is reflexive. Moreover, such a flow can be chosen to satisfy $f(e)\leq E^M$.
\end{proposition}
\begin{proof}
Let $f$ be such a flow and assume that $H$ is a weakly connected component of $G$. Write $f_{ij}$ for the sum of all flows through edges from $i$ to $j$. By finiteness, there is a connected component $C$ in $H$ without ingoing edges. But such a component would also have no outgoing edges by the following calculation:
$$0 = \displaystyle \sum_{i \in C} \sum_{j \in G} f_{ij} - f_{ji} = \displaystyle \sum_{i \in C} \sum_{j \in C} f_{ij} - f_{ji}  + \displaystyle \sum_{i \in C} \sum_{j \in G \backslash C} f_{ij} - f_{ji} 
= 0 + \displaystyle \sum_{i \in C} \sum_{j \in G \backslash C} f_{ij}$$
Hence, $C$ is equal to $H$. 
 
If, conversely, $G$ is a reflexive MD-graph on $n$ vertices, we can consider the set $S$ of all possible cycles on $G$ (no repeated vertices; we allow cycles which use only one edge to go from a vertex back to itself). This set certainly contains at most $E^M$ elements. We obtain a flow $f$ by adding all of these individual cycles in $S$. It is nonzero on all edges since we can complete every directed edge to a cycle by reflexivity. 
\end{proof}

The next graph-theoretic lemma is the key tool in Step (1), since it will allow us to find the distinguished edge $e$.
\begin{lemma}\label{takeaway} Let G be a connected abstract MD-graph with at least one edge. Then there is an edge $e$ such that $G\backslash \{e\}$ is still connected. \end{lemma}
\begin{proof} We prove the claim by induction.

If $|G|=1,2$, the statement holds trivially.

If $|G|>2$, we have $\indegree(v) + \outdegree(v)\geq 3$ for all vertices $v$ in $G$ since $G$ is abstract. We call this the \textit{degree-condition}.
Choose a cycle $C$ of length $k\geq 2$ in $G$. If there is a vertex in $C$ that is joined to any vertex in $C$ apart from its succeeding one, then we can remove an edge without disconnecting the graph. Thus we may also assume that the internal edges of $C$ are exactly the $k$ edges forming the cycle, and hence in particular that $C\neq G$.
Define $G'$ to be the graph obtained from $G$ by contracting $C$ to a single vertex $v$. Then $1 < |G'| < |G|$. At every vertex of $G'$ apart from $v$, the degree-condition holds automatically. 
There is certainly one in- and one outgoing vertex at $v$ by connectedness. 

If this is all, then $k=2$ and the in-/outgoing edges of the cycle must be attached to distinct vertices by the degree condition:
\begin{figure}[H]
	\centering
		\includegraphics[width=0.25\textwidth]{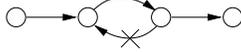}
	\label{fig:possibilities}
		\caption{Special case with circle $C$ of length $k=2$}
\end{figure}
In this case, remove the indicated edge going in the opposite direction without disconnecting the graph.

If, on the other hand, the degree-condition holds at $v$, our smaller contracted graph $G'$ is connected and abstract, and we can remove an edge $e$ without disconnecting $G'$ by induction. Now remove the corresponding edge from $G$. Since every path in $G'$ lifts to a path in $G$, we conclude that also $G\backslash \{e\}$ is connected. \end{proof}

We can now finish the first step of the proof: let $G$ be a reflexive abstract MD-graph with $M$ vertices and $E>0$ edges.
Pick a connected component $C$ of $G$ with at least one edge and remove some edge $e$ from $C$ without disconnecting it by Lemma \ref{takeaway}. Then $G\backslash\{e\}$ is still reflexive, so by Lemma \ref{weaklyconn}, we can find a flow $f'$ on this graph with $0<f'(e')\leq (E-1)^M$ for all edges $e'$. Pick a cycle through $e$ and add the corresponding flow to $f'$ to obtain the flow $f$ required for Step (1).

\subsubsection*{Step (2)}
The next number-theoretic lemma gives a uniqueness result for the scalar product of integral vectors and will facilitate the definition of the edge-weight $w$:
\begin{lemma}\label{numbers}
Given a vector $f=(f_1,...,f_k)\in \NN_0^k$ of nonnegative integers, we can find $(w_1,...,w_k)\in \NN_0^k$ such that:
$$\forall \lambda \in \NN_0^k : \left( \displaystyle \sum_{j=1}^k \lambda_j w_j = \displaystyle \sum_{i=j}^k f_j w_j \right) \Rightarrow \left(\lambda = f \right)$$
and $w_i < 2 (\sum_j f_j+1)^{k+1}$ for all $i$.
\end{lemma}
\begin{proof} Set $M=\displaystyle \sum_{j=1}^k f_j$. Let $n= (M+1)^{k+1}$  and define $w_j$ by $w_j = n+(M+1)^{j-1}$ for $j=1,...,k$.
The result follows by distinguishing the cases $\sum_{j} \lambda_j$ smaller than, larger than, or equal to $M$, and using uniqueness of the $(M+1)-$adic representation in the third case.
\end{proof}

Recall the flow $f$ on $G$ constructed in Step (1). Label all edges other than $e$ by $e_1,...,e_{E-1}$.
We now apply Lemma \ref{numbers} to the vector $(f(e_1),...,f(e_{E-1}))$ to obtain an edge-weight $w$ defined on all edges except for $e$. Give $e=e_E$ the weight $w(e_E)=-\displaystyle \sum_{j=1}^{E-1} f(e_j) w(e_j)$. We have a bound $|w(e_j)|<2 E^{(M+1)(E+1)}$ for all $j$. 

\begin{claim} The flow $f$ is an extremal point in $Q_w=\conv(E_w)$.\end{claim}
\begin{proof} The crucial fact underlying this trick is that if $N$ positive integers sum up to $N$, they must all be equal to $1$. 

Assume that $f\in E_w$. Let $Nf=f_1+...+f_N$ be a decomposition for the flow $Nf$ with $f_1,...,f_N\in E_w$ and $N\in \NN$. Every $f_i$ is a nonzero integral flow in $E_w$ and therefore must have positive flow through $e$ to balance out the negative weight coming from flow through other edges.  The numbers $f_1(e),...,f_N(e)$ are all positive integers and sum up to $N=Nf(e)$. Hence $f_i(e)=1$ for all $i$.

Since $f_i \in E_w$, this implies that the following difference vanishes: $$\displaystyle \sum_{j=1}^{E-1}f_i(e_j)w(e_j) - \sum_{j=1}^{E-1}f(e_j)w(e_j)=0$$
for all $i$, and we conclude $f_i=f$ for all $i$ by the choice of $w$ in \ref{numbers}.
\end{proof}

\subsubsection*{Step (3)}
We will describe hereafter how we can find a flow $g\in V(x)$ such that the graph $\abst(g)$, equipped with the flow induced by $g$ and the edge-weight inherited from $w_x$, is equal to the flowed weighted abstract graph $(G,f,w)$ constructed above. In a second step, we will deduce from Step (2) that $g$ is extremal.

\subsubsection*{Concretising Abstract Graphs}\label{sec:concretise} 
Recall our given vertex-weight $x\in \ZZ^n$. Label the vertices in $G_n$ with $x-$weight $+1$ by $L_1,...,L_p$ (the ``left vertices") and the ones with weight $-1$ by $R_1,...,R_q$ (the ``right vertices").
Assume that $G$ has vertices $V_1,...,V_M$ and edges $e_1,..,e_{E-1},e_E=e$.

Construct a flow $g\in V(x)$ in a step-by-step process as follows: 
The vertex $V_i$ in $G_{n}$ will correspond to $L_i$ in $G$ for $i=1,2,...,M$, so all vertex-representing nodes of $G_{n}$ lie on the left. Having implemented the edges $e_1,..,e_{i-1}$ with the flow-vector $	h_{i-1}$, assume $E_i$ goes from $V_p$ to $V_q$ with $w_ {e_i}=s$. Pick $|s|$ vertices $l_1,...,l_{|s|}$ on the left
and $|s|$ vertices $r_1,...,r_{|s|}$ vertices on the right of $G_{n}$ which do not lie in $\supp(h_{i-1})$. We always have enough vertices available since $p,q\geq M+3E^{3ME+1}\geq M+\sum_{j=1}^E (|w(e_j)|+1)$.
\\
Define a simple path $P$ in $G_{n}$ as follows: 

If $s>0$, consider $P=pr_1l_s...l_1q$. 

If $s=0$, we take $P=pr_1q$.

If $s<0$, the path we use is $P=pr_1r_2...r_{|s|+1}q$.
\\
Obtain the weight $h_i$ from $h_{i-1}$ by adding flow $f(e_i)$ to the edges of the path $P$. One checks easily that the resulting vector $g=h_E$ satisfies $$\abst(g,w_x)=(G,f,w)$$

We are now finally in a position to finish off this proof: Let $g=\sum_j\lambda_j g_j$ be a convex representation of $g$ with $g_j  \in Q_x$, $\lambda_j >0, \sum_j \lambda_j=1$. We can abstract to find flow-vectors $f_j$ on $G$ which induce $g_j$ as in Definition \ref{abst2}. Since the weight $w$ is induced by $w_x$, all abstractions $f_j$ must lie in $Q_w$. We thus obtain a convex representation of the flow $f$ in $Q_w$. By Step (2), we know that $f$ is extremal and hence all flows $f_j$ must be equal to $f$. 

From this, we immediately conclude that $g=g_j$ for all $j$, so $g$ is extremal in $Q_x$. This completes the proof of the Generalized Non-Classifiability Theorem \ref{generalclassi}. \end{proof}

\subsection{Essential Decomposition}
The principal aim of our efforts is to find a concise representation of the $\scl$-polyhedra of elements $z\in M_n$. These are given as:  $$P=P(z)=\conv(D(z)+V(z))=\conv(D(z))+V(z)$$

Here, $V=V(z)$ is the understood recession cone of the polyhedron, and what we need to describe is the contribution of the disc-vectors $D=D(z)$. The set of extremal points of $P$ is the minimal set $S\subset D+V$ with $\conv(S+V)=P$. As a natural alternative, we can consider the minimal set $T\subset D+V$ with $T+V=D+V$, and we obtain the \textit{essential decomposition}. An easy exercise shows that $T$ consists of the following vectors:
\begin{definition} A vector $d\in D$ is an \textit{essential disc-vector} if it cannot be written as a nontrivial sum in $D+V$,
i.e.~if $d=e+v$, $e\in D, v\in V$ implies $v=0$.\end{definition}

It is immediate from the minimality of $S$ that $S$ is contained in $T$, i.e.~that every extremal point is an essential disc-vector. The following example shows that not every essential disc-vector needs to be extremal:
\begin{figure}[htbp]
		\labellist
	\centering
	\pinlabel {\small $1$} at 24 135
\pinlabel {\small $-\!1$} at 356 134
\pinlabel {\small $2$} at 191 248
\pinlabel {\small $4$} at 191 188
\pinlabel {\small $-\!3$} at 188 82
\pinlabel {\small $-\!3$} at 188 22
\pinlabel {$1$} at 116 259
\pinlabel {$1$} at 116 206
\pinlabel {$1$} at 116 114
\pinlabel {$1$} at 116 54

\pinlabel {$1$} at 270 262
\pinlabel {$1$} at 270 206
\pinlabel {$1$} at 270 114
\pinlabel {$1$} at 270 59

\endlabellist		
\includegraphics[width=0.250\textwidth]{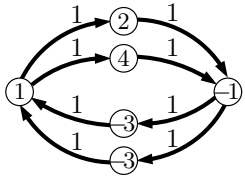}

	\caption{An essential, but not extremal disc-vector (average of two distinct disc-vectors)}
	\label{fig:FIGURE11}
\end{figure}

We have seen in the previous section that the set $S$ of extremal points cannot be classified by the topology of the flows. Since $S\subset T$, this result extends to $T$. There is a further complication of computational nature arising for essential disc-vectors: the essential decomposition of the $\scl$-polyhedra is not suitable for computational purposes. More precisely:

\begin{Essential Membership Theorem} The following decision problem  is \coNP- complete (i.e.~ has \NP-complete complement): ``Given a word $w=\phi(x,y)\in F_2$ of reduced length $m=2n$ and an integral vector $d\in \ZZ^{n^2}$. Is $d$ an essential disc-vector for $D(x)+V(x)$?"
\end{Essential Membership Theorem}

\begin{proof}
The input of the problem can be represented as an element $s=(x,d)$ of $\displaystyle \bigcup_n \left(\ZZ \backslash \{0\}^{n-1}\right) \times \NN_0^{n^2}$, since $y$ is irrelevant and $x_n$ is determined by $x_1,...,x_{n-1}$.

First notice that checking whether a given flow $d\in \NN_0^{n^2}$ has connected support can be done in polynomial time (e.g.~by depth-first search).
To see that the problem is in $\coNP$, assume that the answer to the problem determined by $(x,d)$ is negative, where $x$ has length $n$. 
We can check in polynomial time if $d\in D(x)$, so we may assume that this is the case. 
Given a counterexample $e,v\in \NN_0^{n^2}$, we can check in polynomial time that $e\in D(x)$, $v\in V(x) \backslash \{0\}$ and $d=e+v$. Therefore our problem is in $\coNP$.

To show that it is $\coNP - $complete, we will give a polynomial time reduction from the following \coNP-complete problem:

\begin{problem} \COSUBSETT{}
Given $R\subset \ZZ$ finite, is it true that: $$\forall T\in \mathcal{P}(R)\backslash\{\emptyset\}: \displaystyle \sum_{t\in T} t \neq 0?$$ \end{problem} 

Suppose we have a list of numbers $a_1,...,a_m$ and want to decide the statement: ``No nonempty subset sums up to 0." \\
Set $a_{m+1}=-\displaystyle \sum_{i=1}^m a_i$ and notice that there is a nonempty subset of $\{a_1,...,a_m\}$ summing up to 0 if and only if there is a \textbf{proper} nonempty subset of $\{a_1,...,a_{m+1}\}$ with vanishing sum.
Without loss of generality, it is enough to only consider nontrivial instances of \COSUBSET{}, so to assume $a_k\neq 0$ for all $k$. Let $n=\left( 2(m+1)\right)+2$ and define the vertex weight $x=\left(m,1,a_1,-1,a_2,-1,...,a_{m+1},-1\right)\in M_{n}$.

Consider the vector $d\in D(x)$ determined by the flow drawn below:
\begin{figure}[htbp]
	\centering
		\labellist
\pinlabel {\small $-\!1$} at 404 216.5
\pinlabel {\small $-\!1$} at 489.5 149.7
\pinlabel {\small $-\!1$} at 646 44
\pinlabel {\small $a_1$} at 404 419
\pinlabel {\small $a_2$} at 489 493
\pinlabel {\small $a_{m\!+\!1}$} at 680 604
\pinlabel {\small $1$} at 185 325
\pinlabel {\small $m$} at 44 325

\endlabellist		\includegraphics[width=0.310\textwidth]{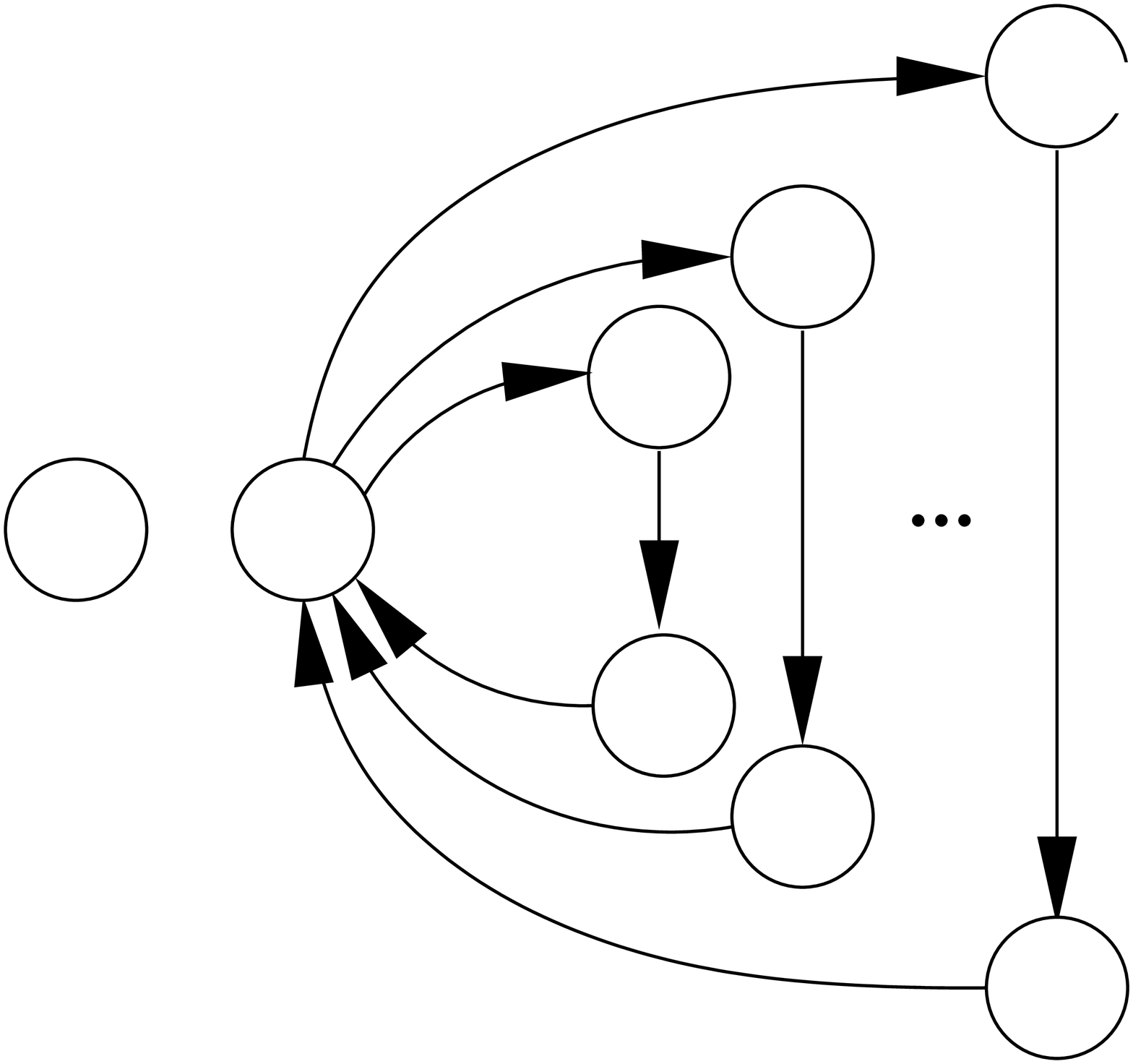}
	\caption{Reduction from \COSUBSET{} to \ESSENTIAL}
	\label{fig:FIGURE10}
\end{figure}

Now $d$ is an essential disc-vector if and only if there is no proper nonzero connected integral subflow $e$ with $h_x(e)=0$, but these flows correspond bijectively to proper nonempty subsets $T\subset\{a_1,...,a_{m+1}\}$ with 
$\displaystyle \sum_{t\in T} a_t = 0$. 
Hence, deciding if $d$ is essential is equivalent to deciding \COSUBSET.
The reduction is computable in polynomial time. \end{proof}
Notice that the above polynomial reduction fails if we restrict ourselves to alternating words since there, the length $m=2n$ of our word grows proportionally to the $|a_i|$ and hence exponentially in the input size.
\\

We conclude the paper with an unrelated, but pretty conjecture we spotted:
\begin{conjecture}
Let $p,q,r\in \NN$ and $n=p+q+r$. Then:
$$\scl(a^{-n} b^{-1} a^p b a^q b^{-1} a^r b)=1-\frac{\gcd(n,q)}{2n}$$
\end{conjecture}

\section{Acknowledgement}
I am very grateful to Prof. Danny Calegari for his advice, suggestions, and stimulating encouragement and to Matthias Goerner, Freddie Manners, Alden Walker and the referee for their helpful comments. I also thank the California Institute of Technology for supporting me with a Summer Undergraduate Research Fellowship. 
\newline

\clearpage{}

\appendix{} 
\fancyhead[EC]{\uppercase{\footnotesize{Freddie Manners}}}
\fancyhead[OC]{\uppercase{\footnotesize{Appendix}}}
%\documentclass{article}
%
%% Packages
%  \usepackage{amsthm}
%  \usepackage{amsmath}
%  \usepackage{amsfonts}
%  \usepackage{enumerate}
%  \usepackage{float}
%  \usepackage{hhline}
%
%% Some ill-advised margin tweaking
%  \pdfpagewidth 8.3in
%  \pdfpageheight 11.7in
%
%  \setlength\topmargin{0in}
%  \setlength\headheight{0in}
%  \setlength\headsep{0in}
%  \setlength\textheight{10.0in}
%  \setlength\textwidth{6.15in}
%  \setlength\oddsidemargin{0.1in}
%  \setlength\evensidemargin{0.1in}

% Theorem styles
  \theoremstyle{plain}
  \newtheorem{thm}{Theorem}
  \newtheorem{lem}[thm]{Lemma}
  \newtheorem{cor}[thm]{Corollary}

  \theoremstyle{definition}
  \newtheorem*{dfn}{Definition}

%% Main document
%
%\title{NP-completeness of a variant of the subset sum problem}
%\author{F.~R.~W.~M.~Manners}
%\begin{document}

\maketitle

\section{(by Freddie Manners) Proof of the NP-completeness of  \MIXEDSUBSETP{}}
\subsection{Definitions}
We recall the definitions of the problems \SUBSET{}, \SUBSETP{}, \VARSUBSETP{}, and \MIXEDSUBSETP{}.  We further define what it means to consider these problems \emph{over $\ZZ^k$}; for example:
\begin{problem}
\SUBSET{} \emph{over $\ZZ^k$}: Given $v_1, \ldots, v_n$ with $v_i \in \ZZ^k$, does there exist a nonzero vector $(\lambda_1, \ldots, \lambda_n) \in \{0,1\}^n$ with $\sum_{j}\lambda_j v_j = 0$?
\end{problem}
The other problems over $\ZZ^k$ are defined analogously.  When we wish to consider the original problem, we may refer to it as e.g.~\SUBSET{} \emph{over $\ZZ$} to avoid ambiguity.

It is a classical fact that \SUBSET{} over $\ZZ$ is NP-complete. We wish to investigate the complexity of \MIXEDSUBSETP{} over $\ZZ$.

\subsection{Proofs}

We first note that \SUBSET{} and \SUBSETP{} are trivially equivalent.

\begin{lem}
  \label{lem:sspnp}
  Given $v_1, \ldots, v_n \in \ZZ$ summing to zero, \SUBSETTP{} holds \emph{iff} \SUBSETT{} holds on $v_1, \ldots, v_{n-1}$.
\end{lem}
\begin{proof}
  Remark that the complement of a solution to \SUBSETP{} (i.e.~setting $\lambda_{i}' = 1 - \lambda_i$) is also a solution.
\end{proof}

We now prove a crucial lemma that allows us to consider solving simultaneous subset sum problems; equivalently, it allows us to work over $\ZZ^k$ rather than $\ZZ$.

\begin{lem}
  \label{lem:parallel}
Each of the above subset sum problems (i.e. \SUBSETP{}, \VARSUBSETP{} or \MIXEDSUBSETP{}) over $\ZZ^k$ has a polynomial reduction to the same problem over $\ZZ$.
\end{lem}
\begin{proof}
  Pick integers $1 = N_1 \ll \ldots \ll N_k$, and set $y_i = \sum_{j = 1}^k N_j v_i^{(j)}$.  Then if the $N_j$ increase sufficiently fast, the problems for $(y_i)$ and $(v_i)$ are equivalent.  (Note the $N_j$ need not increase so rapidly that their lengths in bits are super-polynomial in the other inputs.)
\end{proof}

We can now present our main construction, which aims to reduce \SUBSETP{} (over $\ZZ$) to \MIXEDSUBSETP{} (over $\ZZ^k$ for some $k$). Let $a_1, \ldots, a_n \in \ZZ$ be an input to \SUBSETP{}.  Consider the table:

\begin{table}[H]
  \begin{center}
    \begin{tabular}{|c||*{10}{c|}}
      \hline
      & $\alpha_1$ & $\alpha_2$ & $\ldots$ & $\alpha_n$ & $\beta_1$ & $\beta_2$ & $\ldots$ & $\beta_n$ & $P$ & $Q$ \\
      \hhline{|=#*{10}{=|}}
      $1$ & $a_1$ & $a_2$ & $\ldots$ & $a_n$ & $0$ & $0$ & $\ldots$ & $0$ & $0$ & $0$ \\
      $2$ & $-1$ & $-1$ & $\ldots$ & $-1$ & $-1$ & $-1$ & $\ldots$ & $-1$ & $n$ & $n$ \\
      \hline
      $3$ & $-1$ & $0$ & $\ldots$ & $0$ & $-1$ & $0$ & $\ldots$ & $0$ & $1$ & $1$ \\
      $4$ & $0$ & $-1$ & $\ldots$ & $0$ & $0$ & $-1$ & $\ldots$ & $0$ & $1$ & $1$ \\
      $\vdots$ & $\vdots$ & $\vdots$ & $\vdots$ & $\vdots$ & $\vdots$ & $\vdots$ & $\vdots$ & $\vdots$ & $\vdots$ & $\vdots$ \\
      $n + 2$ & $0$ & $0$ & $\ldots$ & $-1$ & $0$ & $0$ & $\ldots$ & $-1$ & $1$ & $1$ \\
      \hline
      $n + 3$ & $-1$ & $-1$ & $\ldots$ & $-1$ & $0$ & $0$ & $\ldots$ & $0$ & $r$ & $n-r$ \\
      \hline
    \end{tabular}
  \end{center}
\end{table}

Here, the columns beyond the first---labelled by $\alpha_i$, $\beta_i$, $P$, $Q$---represent elements of $\ZZ^k$ for $k = n + 3$, and the rows are simultaneous subset sum problems to be satisfied.  Note that every row sums to zero; that is, $(\alpha_i,\ \beta_i,\ P,\ Q)$ are a valid input to \VARSUBSETP{} or \SUBSETP{}.  Finally, $r$ is some integer in the range $0 < r < n$.

We now suppose that a solution $\lambda = (\lambda_{\alpha_i},\ \lambda_{\beta_i},\ \lambda_P,\ \lambda_Q)$ to \VARSUBSETP{} exists.  We show:
\begin{lem}\mbox{}
  \label{lem:tabletoas}
  \begin{enumerate}[(i)]
    \item $\lambda_P + \lambda_Q = 1$;
    \item $\lambda_{\alpha_i} + \lambda_{\beta_i} = 1$ for all $1 \leq i \leq n$;
    \item $\sum_i \lambda_{\alpha_i} = r \text{ or } n - r$;
    \item $\lambda$ constitutes a solution to \SUBSETP{} on this table;
    \item The $(\lambda_{\alpha_i})$ constitute a solution to \SUBSETP{} on $a_1, \ldots, a_n$.
  \end{enumerate}
\end{lem}
\begin{proof}\mbox{}
  \begin{enumerate}[(i)]
    \item This follows from considering row $2$: if $\lambda_P + \lambda_Q = 0$ then $\sum_j \lambda_j = 0$, and if $\lambda_P + \lambda_Q \geq 2$ then $\sum_j \lambda_j \geq 2n + 2$, which are both forbidden.
    \item This follows from (i) and considering row $i + 2$.
    \item This follows from (i) and considering row $n + 3$.
    \item This follows from (i) and (ii).
    \item That $\sum_i \lambda_{\alpha_i} a_i = 0$ follows from considering row $1$.  Then, recalling that $0 < r < n$, (ii) and (iii) give the other constraints.
  \end{enumerate}
\end{proof}

As a direct consequence of part (iv) of this lemma, we see that the table is a valid input to \MIXEDSUBSETP{}.  We now show a converse to part (v):

\begin{lem}
  \label{lem:astotable}
  If $\mu_1, \ldots, \mu_k$ is a solution to \SUBSETP{} on $a_1, \ldots, a_n$, then for some choice of $r$, we can construct a solution to \SUBSETP{} on the table.
\end{lem}
\begin{proof}
  This is straightforward: take $r = \sum_i \mu_i$, $\lambda_{\alpha_i} = \mu_i$, $\lambda_{\beta_i} = 1 - \mu_i$, $\lambda_P = 1$, and $\lambda_Q = 0$.
\end{proof}

We can now state and prove our main result:

\begin{thm}
  \label{thm:main}
  \SUBSETP{} over $\ZZ$ has a polynomial reduction to \MIXEDSUBSETP{} over $\ZZ^k$ (for $k = n + 3$).
\end{thm}
\begin{proof}
  By Lemmas \ref{lem:tabletoas} and \ref{lem:astotable}, $a_1, \ldots, a_n$ has a solution to \SUBSETP{} \emph{iff} the table has a solution to \MIXEDSUBSETP{} for some value of $r$ ($0 < r < n$).  So, running an oracle for \MIXEDSUBSETP{} at most $n - 2$ times gives a solution to the \SUBSETP{} problem.
\end{proof}

\begin{cor}
  The problem \MIXEDSUBSETP{} over $\ZZ$ is NP-complete.
\end{cor}
\begin{proof}
  We use Lemma \ref{lem:sspnp}, theorem \ref{thm:main} and Lemma \ref{lem:parallel} to give reductions from \SUBSET{} over $\ZZ$, to \SUBSETP{} over $\ZZ$, to \MIXEDSUBSETP{} over $\ZZ^k$, to \MIXEDSUBSETP{} over $\ZZ$ in that order.

  (That \MIXEDSUBSETP{} is in NP is clear.)
\end{proof}

%\end{document}

\clearpage{}

\fancyhead[EC]{\uppercase{\footnotesize{Lukas Brantner}}}

University of Cambridge, St. John's College, Cambridge CB2 1TP, United Kingdom\\\
E-mail address: \href{mailto:dlbb2@cam.ac.uk}{dlbb2@cam.ac.uk}

\end{document}